\documentclass[11pt,a4paper]{article}
\usepackage[T1]{fontenc}
\usepackage[utf8]{inputenc}
\usepackage{lmodern}
\usepackage[margin=1in]{geometry}
\usepackage{amsmath,amssymb,amsthm,mathtools,mathrsfs}
\usepackage{microtype}
\usepackage{bbm}
\usepackage{enumitem,booktabs,tabularx,array}
\usepackage[colorlinks=true,linkcolor=blue,citecolor=blue,urlcolor=blue]{hyperref}
\allowdisplaybreaks
\numberwithin{equation}{section}
\newcommand{\C}{\mathbb C}
\newtheorem{theorem}{Theorem}[section]
\newtheorem{lemma}[theorem]{Lemma}
\newtheorem{proposition}[theorem]{Proposition}
\newtheorem{corollary}[theorem]{Corollary}
\theoremstyle{definition}
\newtheorem{definition}[theorem]{Definition}
\newtheorem{example}[theorem]{Example}
\theoremstyle{remark}
\newtheorem{remark}[theorem]{Remark}
\hypersetup{pdftitle={Structure theory of finite solvable Lie conformal algebras},
	pdfsubject={Cartan inner conjugacy, conformal tori and solvable extensions},pdfauthor={}}
\title{Structure theory of finite solvable Lie conformal algebras}

\author{Bakhrom Omirov, Yuhui Tan}
\date{}
\providecommand{\address}[1]{}
\renewcommand{\address}[1]{%
	\AtEndDocument{%
		\par\bigskip
		{\small\noindent\textsc{#1}\par}%
	}%
}

\providecommand{\email}[1]{}
\renewcommand{\email}[1]{%
	\AtEndDocument{%
		\par\nopagebreak
		{\small\noindent\textit{Email address:} \texttt{#1}\par}%
	}%
}

\address{Bakhrom A. Omirov \newline \indent
	Institute for Advanced Study in Mathematics,
	Harbin Institute of Technology, Harbin 150001, China \newline \indent
	Suzhou Research Institute, Harbin Institute of Technology,
	Suzhou 215104, China}
\email{omirovb@mail.ru}

\address{Yuhui Tan \newline \indent
	Institute for Advanced Study in Mathematics,
	Harbin Institute of Technology, Harbin 150001, China}
\email{26B312001@stu.hit.edu.cn}

\begin{document}

	\maketitle
	\begin{abstract} 
		
		We develop a structure theory of finite solvable Lie conformal algebras. It is proved that every such algebra admits a Cartan subalgebra, and any two Cartan subalgebras are conjugate by a finite product of exponentials of nilpotent zero modes of elements of the derived algebra. Applied to finite vertex algebras, this yields inner conjugacy of Cartan subalgebras, answering a question of D’Andrea and Marchei. The nilradical and solvable radical of an arbitrary finite Lie conformal algebra are shown to be saturated and invariant under ordinary and conformal derivations. For finite free Lie conformal algebras, we prove a simultaneous diagonalization theorem for conformal tori and characterize the existence of a nonzero torus by explicit identities for constant weights. An example demonstrates that maximal conformal tori of a finite nilpotent Lie conformal algebra can have different ranks, so that conjugacy fails in general. The totally saturated null-filiform Lie conformal algebras admitting a nonzero conformal torus are classified. For each algebra in this classification, we prove that maximal conformal tori have rank one and that the maximal solvable extension with the prescribed nilradical is unique (up to isomorphism) and is the semidirect product with a maximal conformal torus.
	\end{abstract}

	
	
	\section{Introduction}\label{in:section}

	Lie conformal algebras encode the singular part of operator product expansions and the commutation relations of local formal distributions \cite{Kac, KacFormal}. Their defining identities combine a bracket depending on a formal parameter with the action of a translation operator, giving rise to structural problems. Finite generation over \(\mathbb C[\partial]\) provides a common framework for studying solvable, nilpotent, and semisimple Lie conformal algebras. Within this framework, D’Andrea and Kac established a conformal analogue of Lie’s theorem for solvable algebras and an analogue of Engel’s theorem characterizing nilpotence through nilpotent adjoint actions. They also classified the finite simple and semisimple Lie conformal algebras \cite{DK}.
	
	The classification of the semisimple objects leaves a substantial part of the general structure problem unresolved. For a finite Lie conformal algebra with nonzero solvable radical, its structure also depends on the radical and on the extension through which it is attached to the
	semisimple quotient. In finite-dimensional complex Lie theory, Levi's theorem separates these questions by expressing every algebra as a semidirect product of its solvable radical and a semisimple subalgebra \cite{Bourbaki2005}. The conformal analogue of this theorem fails: the extension of the semisimple quotient by the solvable radical need not split \cite{DK}. Consequently, a structure theory of finite Lie conformal algebras must account for solvable radicals and the extension as well as semisimple quotients. Understanding solvable Lie conformal algebras is therefore a necessary part of the general
	structure problem.
	
	There is a another motivation for the investigation of solvable Lie conformal algebras coming from vertex algebras. D'Andrea
	proved that the Lie conformal algebra underlying every finite vertex algebra is solvable \cite{AD2008}. Thus, in this setting, the solvable case describes the entire underlying Lie conformal structure. Structural results for solvable Lie conformal algebras can accordingly be applied to the decompositions and automorphisms of finite vertex algebras. The inner conjugacy theorem proved here provides one such application, resolving the question raised by D'Andrea and Marchei about Cartan subalgebras \cite[\S5.4]{DM2012}.
	
	Classical solvable Lie theory suggests studying these problems through the nilradical and its derivations. Cartan subalgebras organize the adjoint representation, while tori of derivations describe diagonal actions relevant to extensions with a prescribed nilradical \cite{Bourbaki2005,OS, Snobl2010}. In the conformal setting, this approach requires additional structural restrictions. A quotient of a free $\mathbb{C}[\partial]$-module may have torsion, and a proper submodule may possess the same rank as the ambient module. Compositions of conformal
	endomorphisms involve shifts in the translation variable; moreover, an element can generate a nonabelian conformal subalgebra. The interaction between the module structure and the bracket must therefore be incorporated into the arguments for decompositions, conjugacy, and extensions. Existing classifications in low rank, including that of finite free Lie conformal algebras of rank two \cite{BCH}, illustrate the interaction between the conformal bracket and the underlying module structure. These results motivate the development of structural principles for solvable and nilpotent Lie conformal algebras of arbitrary finite rank.
	
	In this work we develop our approach along three related lines: the existence and inner conjugacy of Cartan subalgebras, the structure of conformal tori, and solvable extensions with a prescribed nilradical. Together, these results establish conformal analogues of several central principles of classical solvable Lie theory and identify a failure of torus conjugacy that distinguishes the two settings.
	
	For an arbitrary finite Lie conformal algebra, both the nilradical and the solvable radical are saturated and invariant under ordinary and conformal derivations. Saturation ensures that the corresponding quotients are free over $\C[\partial]$, so that module complements exist. We also show that an ordinary or conformal derivation of a finite nilpotent Lie conformal algebra $\mathcal N$ is nilpotent if and only if its induced action on the abelianization $\mathcal N/\mathcal N^2$ is nilpotent. These properties make it possible to study solvable extensions through nil-independent conformal derivations of their nilradicals.
	
	Our Cartan theorem holds for every finite solvable Lie conformal
	algebra, without a freeness assumption. Such an algebra admits a
	Cartan subalgebra, and any two Cartan subalgebras are conjugate by a finite product of exponentials of nilpotent zero modes of elements of the derived algebra. Applied to a finite vertex algebra, these zero modes are derivations of the vertex algebra itself. Their exponentials therefore give inner vertex algebra automorphisms, strengthening the automorphism conjugacy established in \cite{DM2012} to inner conjugacy.
	
	For a finite free Lie conformal algebra $\mathcal L$, a
	\emph{conformal torus} is an abelian conformal subalgebra of
	$\operatorname{CDer}(\mathcal L)$ consisting of diagonalizable
	conformal endomorphisms. We prove simultaneous diagonalizability
	and show that every maximal conformal torus has a basis acting
	by constant diagonal matrices in a common basis of $\mathcal L$.
	This reduces the existence of a nonzero conformal torus to explicit linear identities for constant weights. Besides the usual weight addition rule, these identities constrain the dependence of the structure polynomials on the translation and bracket variables. In a basis diagonalizing a maximal torus, the dimension of their solution space equals the rank of that torus. Nevertheless, maximal conformal tori need not be conjugate: we construct a finite nilpotent Lie conformal algebra with maximal tori of different ranks.
	
	In this paper we also study finite free nilpotent Lie conformal algebras of maximal nilindex. A nilpotent Lie algebra of dimension $n$ has nilindex at most $n$, since its abelianization has dimension at least two. In the conformal setting, however, there exist nilpotent Lie conformal algebras of rank $n$ with nilindex $n+1$; see \cite{SYY}. Following the terminology of Leibniz algebra theory, we call such algebras \emph{null-filiform}. These algebras provide conformal counterparts to the nilpotent Lie algebras of maximal nilpotency class studied in \cite{CabezasPastor,CaglieroTirao,GomezJimenez,Vergne1970}. 
	We focus our approach to the class of totally saturated null-filiform Lie conformal algebras admitting a nonzero torus. Total saturation means that every term of the lower central series is saturated; its nonzero successive quotients are then free of rank one. These conditions permit a basis adapted to the lower central series, in which the torus weight identities impose strong restrictions on the brackets. We obtain the families $\mathcal A_n$, $\mathcal C_n(\beta)$, and $\mathcal D_n(\gamma)$ and determine their isomorphism classes. For every algebra in this classification, maximal conformal tori have rank one and are conjugate by the exponential of a nilpotent ordinary derivation.
	
	For these nilradicals, the structural results also determine the
	maximal solvable extension. If $\mathcal N$ is totally saturated
	and null-filiform and admits a nonzero conformal torus, every finite solvable Lie conformal algebra $\mathcal R$ with nilradical $\mathcal N$ satisfies $\operatorname{rank}(\mathcal R/\mathcal N)\leq1$ and embeds into 
	$$\mathcal R_{\mathcal N}=\C[\partial]d\ltimes\mathcal N,
	\qquad d_\lambda e_i=i e_i,\quad 1\leq i\leq n,$$
	where $\{e_1, \ldots, e_n\}$ is a suitable adapted basis of
	$\mathcal N$ and the embedding maps $\mathcal N$ onto itself.
	Hence $\mathcal R_{\mathcal N}$ is, up to isomorphism, the unique
	extension maximal under embeddings preserving the prescribed
	nilradical. In this class, Cartan theory and the conjugacy of
	maximal conformal tori thus lead from a classification of
	nilpotent algebras to a determination of their maximal solvable
	extensions.
	
	The paper is organized as follows. Section~2
	recalls the necessary facts about Lie conformal algebras, conformal modules, and generalized weights. Section~3 studies derivations and proves the invariance of radicals and the nilpotence criterion on the abelianization. Section~4 establishes
	Cartan existence and inner conjugacy, including the application
	to finite vertex algebras. Section~5 develops the theory of conformal tori and gives the example of maximal tori of different ranks. Section~6 proves the null-filiform classification and conjugacy of maximal tori within that class. Finally, Section~7 establishes saturation of radicals and studies solvable extensions, proving the maximal-extension theorem for the classified nilradicals.
	
	Throughout the paper, all modules are assumed to be finitely generated over $\C[\partial]$, unless otherwise stated. We write $\operatorname{rank}(\mathcal M)$ for the rank of a module
	$\mathcal M$ over $\C[\partial]$ and use the abbreviation \emph{LCA} for Lie conformal algebra. Unless otherwise specified, $\mathcal L$, $\mathcal R$, and $\mathcal N$ denote an arbitrary
	finite LCA, a solvable finite LCA, and a nilpotent finite LCA, respectively.
	
	\section{Preliminaries}\label{pr:section}
	
	This section recalls the basic definitions and structural results on Lie conformal algebras and their modules, including generalized weights and the associated decompositions; see \cite{BDK,ChengKac,DK,DM2012}.
	
	\begin{definition}\label{def:lie-conformal-algebra}
		An \emph{LCA} $\mathcal{L}$ is a module endowed with a $\C$-bilinear $\lambda$-bracket
		$\mathcal{L}\otimes \mathcal{L}\rightarrow \mathcal{L}[\lambda]$ as $a\otimes b\longmapsto [a_\lambda b],$ satisfying the following identities,
		\begin{align*}
			[\partial a_\lambda b]&=-\lambda[a_\lambda b],\quad [a_\lambda\partial b]=(\partial+\lambda)[a_\lambda b],&&\text{\emph{(sesquilinearity)}},\\
			[a_\lambda b]&=-[b_{-\partial-\lambda}a],&&\text{\emph{(skew-symmetry)}},\\
			[[a_\lambda b]_{\lambda+\mu}c]&=[a_\lambda[b_\mu c]]-[b_\mu[a_\lambda c]],&&\text{\emph{(Jacobi identity)}}.
		\end{align*}
	\end{definition}
	
	For submodules $\mathcal{U}$ and $\mathcal{V}$ of $\mathcal{L}$ we denote by $[\mathcal{U}\cdot \mathcal{V}]$ the submodule of $\mathcal{L}$ generated by all coefficients in $\lambda$ of the $\lambda$-brackets $[u_\lambda v]$, where $u\in \mathcal{U}$ and $v\in \mathcal{V}$.
	
	The lower central and the derived series of $\mathcal{L}$ are, respectively,
	$$\mathcal{L}^1=\mathcal{L},\quad \mathcal{L}^{i+1}=[\mathcal{L}\cdot \mathcal{L}^i];\qquad \mathcal{L}^{(1)}=\mathcal{L},\quad \mathcal{L}^{(i+1)}=[\mathcal{L}^{(i)}\cdot \mathcal{L}^{(i)}].$$
	
	An algebra $\mathcal L$ is called \emph{nilpotent} if
	$\mathcal L^n=0$ for some $n\in \mathbb{N}$. The minimal such $n$ is called the
	\emph{nilindex} of $\mathcal L$. Similarly, $\mathcal L$ is called \emph{solvable} if $\mathcal L^{(m)}=0$ for some $m\in \mathbb{N}$, and the minimal such $m$ is called the \emph{derived length} of $\mathcal L$.

	We shall use the standard facts concerning nilpotent and solvable ideals. We have 
	$$[\mathcal L^i\cdot\mathcal L^j]\subseteq\mathcal L^{i+j},\qquad\mathcal L^{(i)}\subseteq\mathcal L^{2^{i-1}},\qquad\bigl(\mathcal L^{(i)}\bigr)^{(j)}=\mathcal L^{(i+j-1)}.$$
	
	As in the case of classical Lie algebras, any finite sum of nilpotent (respectively, solvable) ideals is nilpotent (respectively, solvable). Since $\mathcal L$ is a Noetherian module, the sum of all nilpotent ideals, respectively all solvable ideals, is a finite sum. Hence there exist unique maximal nilpotent and solvable ideals of $\mathcal L$, called the \emph{nilradical} and the \emph{solvable radical} and denote them by $\operatorname{Nil}(\mathcal L)$ and $\operatorname{Rad}(\mathcal L),$ respectively.
	
	For a submodule $\mathcal{U}\subseteq \mathcal{M}$ we define the {\it saturation} of $\mathcal{U}$ and the {\it torsion} of $\mathcal{M}$ as follows:
	$$\begin{array}{llll}
		\operatorname{Sat}_{\mathcal{M}}(\mathcal{U})&=&\{m\in \mathcal{M}\mid p m\in \mathcal{U}\text{ for some } 0\ne p\in \C[\partial]\},\\[3mm]
		\operatorname{Tor}(\mathcal {M})&=&\{m\in \mathcal {M} \mid pm=0 \text{ for some } 0\neq p\in\mathbb{C}[\partial]\}.\\
	\end{array}
	$$
	Note that due to \cite{DK} we have  $\operatorname{Tor}(\mathcal M)\subseteq \operatorname{Center}(\mathcal M)$.
	
	Note that saturation preserves rank: for every submodule $\mathcal U$ of $\mathcal L$, one has 
	$\operatorname{rank}\bigl(\operatorname{Sat}_{\mathcal L}(\mathcal U)\bigr)=\operatorname{rank}(\mathcal U).$ 
	
	A submodule $\mathcal{U}\subseteq\mathcal{M}$ is \emph{saturated} if $\mathcal{U}=\operatorname{Sat}_{\mathcal{M}}(\mathcal{U})$, or equivalently, if $\mathcal{M}/\mathcal{U}$ is torsion-free.
	Since $\C[\partial]$ is a PID, every finite torsion-free module is free. Hence
	$$\mathcal L\cong_{\C[\partial]}\operatorname{Tor}(\mathcal L)\oplus\bigl(\mathcal L/\operatorname{Tor}(\mathcal L)\bigr).$$
	
	\begin{definition}
		A \emph{conformal $\mathcal{L}$-module} is a module $\mathcal{M}$
		equipped with a $\C$-bilinear action
		$(a,m)\mapsto a_\lambda m\in\mathcal{M}[\lambda]$ satisfying
		$$(\partial a)_\lambda m=-\lambda a_\lambda m,\qquad a_\lambda(\partial m)=(\partial+\lambda)a_\lambda m, \qquad a_\lambda(b_\mu m)-b_\mu(a_\lambda m)=[a_\lambda b]_{\lambda+\mu}m$$
		for all $a,b\in\mathcal{L}$ and $m\in\mathcal{M}$.
	\end{definition}
	
	Let $\mathcal{M}$ be a finite conformal $\mathcal{L}$-module.
	A $\C$-linear map $\phi\colon \mathcal{L}\rightarrow \C[\lambda],$ $a\mapsto \phi_a(\lambda),$ 
	is called a \emph{weight} of $\mathcal{M}$ if there exists an 
	element $0\neq m\in\mathcal{M}$ such that $a_\lambda m=\phi_a(\lambda)m$ for all $a\in\mathcal{L}.$ Any such element $m$ is called a \emph{weight vector} of weight $\phi$.
	
	\begin{definition} \label{def2.3}  Let $\phi$ be a weight of $\mathcal{M}$. Define an increasing sequence of $\C[\partial]$-submodules recursively by
		$$\mathcal{M}^\phi_{-1}=0, \quad \mathcal{M}^\phi_{j+1}
		=\operatorname{Span}_{\C[\partial]} 
		\left\{m\in\mathcal{M}\;\middle|\;
		a_\lambda m-\phi_a(\lambda)m\in\mathcal{M}^\phi_j[\lambda]
		\text{ for all }a\in\mathcal{L}\right\},
		\quad j\geq -1.
		$$
		The submodule of $\mathcal{M}$
		$$\mathcal{M}^\phi:=\bigcup_{j\geq -1}\mathcal{M}^\phi_j$$
		is called the \emph{generalized weight submodule} of $\mathcal{M}$
		associated with $\phi$.
	\end{definition}
	
	Let $\mathcal{M}$ be a conformal $\mathcal{R}$-module that is free of rank $r$. By the conformal analogue of Lie's theorem
	\cite{DK}, $\mathcal{M}$ admits a saturated flag of conformal
	$\mathcal{R}$-submodules
	\begin{equation}\label{ca:triangular-character}
		\begin{gathered}
			0=\mathcal{M}_0\subset\mathcal{M}_1\subset\cdots
			\subset\mathcal{M}_r=\mathcal{M}, \qquad
			\mathcal{M}_i/\mathcal{M}_{i-1}=\C[\partial]\bar m_i\\
			a_\lambda\bar m_i=\phi_i(a)\bar m_i, \qquad
			\phi_i:\mathcal{R}\to\C[\lambda],
		\end{gathered}
	\end{equation}
	where each $\C$-linear map $\phi_i\colon\mathcal{R}\to\C[\lambda]$
	vanishes on $\mathcal{R}^2$.
	For the adjoint module $\mathcal{M}=\mathcal{R}$, the terms
	$\mathcal{M}_i$ are ideals that satisfies $[\mathcal{R}^2{}_\lambda\mathcal{M}_i]
	\subseteq\mathcal{M}_{i-1}[\lambda].$
	
	Let $\mathcal M$ be a finite conformal $\mathcal N$-module.
	By \cite{BDK,DM2012}, $\mathcal M$ admits the generalized
	weight decomposition
	
	\begin{equation}\label{ca:weight-decomposition}
		\mathcal{M}=\bigoplus_{\phi\in\operatorname{W}(\mathcal{M})}\mathcal{M}^{\phi},
	\end{equation}
	where $\operatorname{W}(\mathcal{M})$ denotes the
	set of weights of $\mathcal{M}$ and  its cardinality is bounded above by the minimal number of generators of $\mathcal{M}$.
	
	Recall, a \emph{conformal endomorphism} of $\mathcal{L}$ is a $\C$-linear map $D_\lambda:\mathcal{L}\to\mathcal{L}[\lambda]$ satisfying $D_\lambda(\partial a)=(\partial+\lambda)D_\lambda a$ for all $a\in\mathcal{L}$. The space $\operatorname{Cend}(\mathcal{L})$ of all such maps is a module under the action $(\partial D)_\lambda=-\lambda D_\lambda$.
	
	We recall the following two notions. A \emph{conformal derivation} of $\mathcal L$ is an element $D\in\operatorname{Cend}(\mathcal L)$, while a \emph{derivation} of $\mathcal L$ is an element $d\in\operatorname{End}_{\mathbb{C}[\partial]}(\mathcal L)$, satisfying, respectively,
	$$D_\lambda[a_\mu b]=[(D_\lambda a)_{\lambda+\mu}b]+[a_\mu D_\lambda b],\qquad
	d[a_\lambda b]=[d(a)_\lambda b]+[a_\lambda d(b)].$$
	Their spaces are denoted by $\operatorname{CDer}(\mathcal L)$ and $\operatorname{Der}(\mathcal L)$, respectively.
	
	A conformal derivation is called \emph{inner} if it equals $\operatorname{ad}_x$ for some $x\in\mathcal{L}$, where $(\operatorname{ad}_x)_\lambda y=[x_\lambda y]$ for all $y\in\mathcal{L}$. The module of inner conformal derivations is denoted by $\operatorname{CInn}(\mathcal{L})$.
	For $D\in\operatorname{CDer}(\mathcal{L})$, evaluation at $\lambda=0$ yields $D_0\in\operatorname{Der}(\mathcal{L})$. A conformal endomorphism $D$ is called \emph{nilpotent} if $D_{\lambda_m}\cdots D_{\lambda_1}=0$ for some $m$, where $\lambda_1,\ldots,\lambda_m$ are independent parameters.
	
	\section{Radical invariance and nilpotence of derivations}
	
	Conformal derivations and their generalizations have been
	studied for various classes of Lie conformal algebras
	\cite{FanHongSu,Hong,LeiTan}. This section establishes the invariance of radicals under ordinary and conformal derivations and a criterion for detecting nilpotence of derivations on the abelianization of a nilpotent Lie conformal algebra.
	
	\begin{lemma}\label{pr:derivation-nilpotent-ideal}
		Let $\mathcal I$ be a nilpotent ideal of $\mathcal L$. Then, for every
		$D\in\operatorname{CDer}(\mathcal L)$ and every $\alpha\in\mathbb C$,
		the sum $\mathcal I+D_\alpha(\mathcal I)$ is a nilpotent ideal of
		$\mathcal L$, where $D_\alpha={D_\lambda}_{ | \lambda=\alpha}$.
	\end{lemma}
	
	\begin{proof}
		Set $\varphi=D_\alpha$ and $\mathcal J=\mathcal I+\varphi(\mathcal I)$.
		The equality $\partial\varphi(i)=\varphi(\partial i)-\alpha\varphi(i)$ for all $i\in\mathcal I,$ shows that $\mathcal J$ is a $\mathbb C[\partial]$-submodule of
		$\mathcal L$. Moreover, the conformal derivation identity yields
		\[
		[x_\lambda\varphi(i)]
		=\varphi([x_\lambda i])-[\varphi(x)_{\lambda+\alpha}i]
		\in\mathcal J[\lambda],
		\qquad x\in\mathcal L,\quad i\in\mathcal I.
		\]
		Thus $\mathcal J$ is an ideal of $\mathcal L$.
		
		Let $\mathcal A$ be the associative subalgebra of
		$\operatorname{End}_{\mathbb C}(\mathcal L)$ generated by
		$\bigl\{(\operatorname{ad}_x)_\beta \ | \ x\in\mathcal L,\ \beta\in\mathbb C\bigr\},$
		and let $\mathcal K$ be the two-sided ideal of $\mathcal A$ generated by those operators for which $x\in\mathcal I$. 
		Since $\mathcal K(\mathcal L)\subseteq\mathcal I$ and $\mathcal K(\mathcal I^j)\subseteq\mathcal I^{j+1}$ for all $j \geq 1,$
		we obtain $\mathcal K^m=0$, where $m$ is the nilindex of $\mathcal I.$ 
		
		The conformal derivation identity also gives $[\varphi,(\operatorname{ad}_x)_\beta]
		=(\operatorname{ad}_{\varphi(x)})_{\beta+\alpha}
		\in\mathcal A$,
		which implies that $[\varphi,\mathcal A]\subseteq\mathcal A$, and
		${\delta=[\varphi,-]}_{|\mathcal A}$ is a derivation of the associative algebra $\mathcal A$. Then $\mathcal B=\mathcal K+\delta(\mathcal K).$
		
		For $k_1,\ldots,k_m\in\mathcal K$ the iterated Leibniz rule implies
		\[
		\begin{aligned}
			0
			&=\delta^m(k_1\cdots k_m)\\
			&=\sum_{\substack{r_1+\cdots+r_m=m\\ r_1,\ldots,r_m\geq0}}
			\frac{m!}{r_1!\cdots r_m!}\,
			\delta^{r_1}(k_1)\cdots\delta^{r_m}(k_m)\\
			&\equiv m!\,\delta(k_1)\cdots\delta(k_m)
			\pmod{\mathcal K}.
		\end{aligned}
		\]
		Therefore, we obtain
		$\mathcal B^m\subseteq\mathcal K$, and hence $\mathcal B^{m^2}=0.$
		
		Finally, for every $i\in\mathcal I$ and $\beta\in\mathbb C$, we have $(\operatorname{ad}_{\varphi(i)})_\beta
		=\delta\bigl((\operatorname{ad}_i)_{\beta-\alpha}\bigr)
		\in\mathcal B.$ Hence, $(\operatorname{ad}_x)_\beta\in\mathcal B$ for all
		$x\in\mathcal J$ and $\beta\in\mathbb C$. It follows that, for all
		$x_1,\ldots,x_{m^2+1}\in\mathcal J$,
		\[
		[x_1{}_{\lambda_1}
		[x_2{}_{\lambda_2}\cdots
		[x_{m^2}{}_{\lambda_{m^2}}x_{m^2+1}]
		\cdots]]=0,
		\]
		because this polynomial vanishes under every scalar specialization
		of $\lambda_1,\ldots,\lambda_{m^2}$. Thus $\mathcal J^{m^2+1}=0$,
		proving that $\mathcal J$ is nilpotent.
	\end{proof}
	
	The following result shows that both the nilradical and the solvable radical are invariant under derivations and conformal derivations. 
	
	\begin{proposition}\label{pr:derivation-radicals}
		The nilradical $\operatorname{Nil}(\mathcal{L})$ and the solvable radical
		$\operatorname{Rad}(\mathcal{L})$ are invariant under every derivation
		and every conformal derivation of $\mathcal{L}$.
	\end{proposition}
	
	\begin{proof}
		Set $\mathcal{N}=\operatorname{Nil}(\mathcal{L})$ and
		$\mathcal{R}=\operatorname{Rad}(\mathcal{L})$, and let
		$D\in\operatorname{CDer}(\mathcal{L})$.
		Lemma~\ref{pr:derivation-nilpotent-ideal}, together with the maximality of $\mathcal{N}$, implies that
		$D_\alpha(\mathcal{N})\subseteq\mathcal{N}$ for every $\alpha$.
		Consequently, $D_\lambda(\mathcal{N})\subseteq\mathcal{N}[\lambda].$
		
		We prove that $D$ also preserves $\mathcal{R}$ by induction on its derived length $s$. If $s=1, 2$, then $\mathcal{R}$ is abelian, so $\mathcal{R}=\mathcal{N}$ and the assertion follows.
		Suppose now that $s > 2$. The ideal $\mathcal{R}^{(s-1)}$ is abelian and hence contained in $\mathcal{N}$. Thus, we get $(\mathcal{R}/\mathcal{N})^{(s-1)}=0.$ Moreover, we have $\operatorname{Rad}(\mathcal{L}/\mathcal{N})
		=\mathcal{R}/\mathcal{N}.$ Since $D$ preserves $\mathcal{N}$, it induces a conformal derivation of $\mathcal{L}/\mathcal{N}$. By the induction hypothesis, the induced derivation preserves $\mathcal{R}/\mathcal{N}$, whence $D_\lambda(\mathcal{R})\subseteq\mathcal{R}[\lambda].$
		
		Finally, let $d\in\operatorname{Der}(\mathcal{L})$.
		The argument of Lemma~\ref{pr:derivation-nilpotent-ideal} applies
		equally to ordinary derivations and shows that
		$\mathcal{N}+d(\mathcal{N})$ is a nilpotent ideal.
		The maximality of $\mathcal{N}$ therefore gives
		$d(\mathcal{N})\subseteq\mathcal{N}$.
		The same induction on the derived length of $\mathcal{R}$ then yields $d(\mathcal{R})\subseteq\mathcal{R}$.
	\end{proof}
	
	Let $D$ be a (conformal) derivation of $\mathcal L$.
	By Proposition~\ref{pr:derivation-radicals}, $D$ preserves
	$\mathcal N=\operatorname{Nil}(\mathcal L)$, and its restriction
	$\widetilde D=D_{|\mathcal N}$ is a (conformal) derivation
	of $\mathcal N$. The (conformal) derivation identity implies that $\widetilde D$ preserves $\mathcal N^2$.
	
	Denote by $\widetilde D_{11}$ the induced (conformal) linear map
	on $\mathcal N/\mathcal N^2$.
	The following proposition shows that nilpotence of
	$\widetilde D$ is equivalent to nilpotence of
	$\widetilde D_{11}$.
	In the conformal case, nilpotence is understood with respect
	to iterated conformal composition with independent formal
	parameters.
	
	\begin{proposition}\label{pr:nilpotence-quotient}
		A (conformal) derivation $\widetilde D$ of $\mathcal N$ is
		nilpotent if and only if $\widetilde D_{11}$ is nilpotent.
	\end{proposition}
	\begin{proof}
		The (conformal) derivation identity shows that $\mathcal N^2$
		is invariant under $\widetilde D$. Thus $\widetilde D^{\,m}=0$
		implies $\widetilde D_{11}^{\,m}=0$.
		
		Conversely, suppose that $\widetilde D_{11}^{\,m}=0$.
		Then $\widetilde D^{\,m}(\mathcal N)\subseteq\mathcal N^2$.
		Applying the iterated derivation argument in
		\cite{Snobl2010} gives
		$\widetilde D^{\,i(m-1)+1}(\mathcal N^i)\subseteq\mathcal N^{i+1},$ $i\geq1,$ where conformal compositions use independent formal parameters. Since the lower central series terminates, $\widetilde D$ is nilpotent.
	\end{proof}

	\section{Cartan subalgebras of finite solvable Lie conformal algebras.} \label{ca:section}
	
	In this section we establish the existence and inner conjugacy of Cartan subalgebras of finite solvable Lie conformal algebras,
	with an application to finite vertex algebras.

	\begin{proposition}\label{ca:weight-exactness}
		For each weight $\phi$, the functor $\mathcal M\mapsto\mathcal M^\phi$ on finite conformal $\mathcal N$-modules is exact. More explicitly, every short exact sequence
		\[
		0\longrightarrow \mathcal M'
		\xrightarrow{i}\mathcal M
		\xrightarrow{q}\mathcal M''
		\longrightarrow 0
		\]
		induces a short exact sequence
		\[
		0\longrightarrow (\mathcal M')^\phi
		\xrightarrow{i_{| (\mathcal M')^\phi}}\mathcal M^\phi
		\xrightarrow{q_{| \mathcal M^\phi}}(\mathcal M'')^\phi
		\longrightarrow 0.
		\]
		Equivalently, 
		\begin{equation}\label{ca:weight-exact-equation}
			(\mathcal M')^\phi=\mathcal M'\cap\mathcal M^\phi,
			\qquad
			q(\mathcal M^\phi)=(\mathcal M'')^\phi.
		\end{equation}
	\end{proposition}
	\begin{proof} We first show that every homomorphism $f:\mathcal U\to\mathcal V$ of conformal $\mathcal N$-modules satisfies
		\begin{equation}\label{eq4.21}
			f(\mathcal U_j^\psi)\subseteq\mathcal V_j^\psi, \quad j\geq-1.
		\end{equation}
		The assertion is immediate for $j=-1$. Suppose it holds for $j$. If $u$ is one of the generators defining $\mathcal U_{j+1}^\psi$, then 
		$$a_\lambda f(u)-\psi_a(\lambda)f(u)=f\bigl(a_\lambda u-\psi_a(\lambda)u\bigr)\in\mathcal V_j^\psi[\lambda] \quad \mbox{for every} \quad a\in\mathcal N.$$
		Hence $f(u)\in\mathcal V_{j+1}^\psi$. Since $f$ is $\C[\partial]$-linear, the same inclusion holds for the submodule generated by these elements. Induction therefore gives $f(\mathcal U^\psi)\subseteq\mathcal V^\psi$.
		
		Identify $\mathcal M'$ with its image under $i$. Applying \eqref{eq4.21} to $i$ gives $(\mathcal M')^\phi\subseteq\mathcal M'\cap\mathcal M^\phi$. Conversely, let $m'\in\mathcal M'\cap\mathcal M^\phi$. By \eqref{ca:weight-decomposition}, decompose $m'$ in $\mathcal M'$ as
		$m'=\sum\limits_\psi m'_\psi,$ where $m'_\psi\in(\mathcal M')^\psi.$
		Each $m'_\psi$ belongs to $\mathcal M^\psi$. Since $m'\in\mathcal M^\phi$, the weight decomposition of $\mathcal M$ gives $m'_\psi=0$ for $\psi\ne\phi$. Thus $m'=m'_\phi\in(\mathcal M')^\phi$, proving
		$(\mathcal M')^\phi=\mathcal M'\cap\mathcal M^\phi.$
		
		Now applying \eqref{eq4.21} to $q$ gives $q(\mathcal M^\phi)\subseteq(\mathcal M'')^\phi$. For the reverse inclusion, take $m''\in(\mathcal M'')^\phi$ and choose $m\in\mathcal M$ with $q(m)=m''$. Write $m=\sum\limits_\psi m_\psi,\ m_\psi\in\mathcal M^\psi.$
		Then $m''=\sum\limits_\psi q(m_\psi),$ where $q(m_\psi)\in(\mathcal M'')^\psi.$ By directness of the weight decomposition of $\mathcal M''$, we have $q(m_\phi)=m''$ and $q(m_\psi)=0$ for $\psi\ne\phi$. Consequently, $q(\mathcal M^\phi)=(\mathcal M'')^\phi$.
		
		Finally, $\ker(q_{|\mathcal M^\phi})=\mathcal M'\cap\mathcal M^\phi=(\mathcal M')^\phi.$ Together with the injectivity of $i$, this proves exactness.
	\end{proof}
	
	Let $\mathcal M$ be a finite conformal $\mathcal N$-module, with action $\rho(a)_\lambda m=a_\lambda m$. For $a\in\mathcal N$, set
	$$\mathcal M_{0,a}=\bigcup_{n\geq1}\ker\rho(a)^{n},
	\quad \mbox{where} \quad \rho(a)^{n}(m):=a_{\lambda_n}\cdots a_{\lambda_1}m.$$
	Definition~\ref{def2.3} and the generalized weight decomposition~\eqref{ca:weight-decomposition} then yield
	\begin{equation}\label{eq4.3}
		\mathcal M_{0,a}=\bigoplus\limits_{\phi_a=0}\mathcal M^\phi
		\quad \mbox{for every} \quad a\in\mathcal N.
	\end{equation}
	
	\subsection{The existence of Cartan subalgebra.}
	
	Recall that a \emph{Cartan subalgebra} of $\mathcal L$ is a conformal subalgebra
	$\mathcal H$ satisfying
	$$\mathcal H \quad \mbox{is nilpotent}\quad \mbox{and} \quad N_{\mathcal L}(\mathcal H)=\mathcal H,$$
	where $N_{\mathcal L}(\mathcal H)=\{x\in\mathcal L\mid [x_\lambda\mathcal H]\subseteq\mathcal H[\lambda]\}$ is the \emph{normalizer} of $\mathcal H$ in $\mathcal L$.
	
	\begin{lemma}\label{ca:normalizer-facts}
		Let $\mathcal H$ be a Cartan subalgebra of an algebra $\mathcal L$. Then
		\begin{itemize}
			\item[(i)] $\operatorname{Center}(\mathcal L)\subseteq\mathcal H$;
			\item[(ii)] $\operatorname{Sat}_{\mathcal L}(\mathcal H)=\mathcal H$;
			\item[(iii)] if $\mathcal K$ is a nilpotent conformal subalgebra of $\mathcal L$ properly containing $\mathcal H$, then
			$\mathcal H\subsetneq N_{\mathcal K}(\mathcal H)$.
		\end{itemize}
		In particular, $\mathcal H$ is maximal among the nilpotent conformal subalgebras of $\mathcal L$.
	\end{lemma}
	\begin{proof}
		Since every central element normalizes $\mathcal H$, \textup{(i)} follows from $N_{\mathcal L}(\mathcal H)=\mathcal H$. For \textup{(ii)}, suppose that $0\neq p(\partial)x\in\mathcal H$. Sesquilinearity gives $p(-\lambda)[x_\lambda\mathcal H]\subseteq\mathcal H[\lambda],$ whence $[x_\lambda\mathcal H]\subseteq\mathcal H[\lambda]$. Thus $x\in N_{\mathcal L}(\mathcal H)=\mathcal H$. For \textup{(iii)}, nilpotence provides $j\geq2$ such that $\mathcal K^j\subseteq\mathcal H$ but $\mathcal K^{j-1}\not\subseteq\mathcal H$. Then $[\mathcal H\cdot\mathcal K^{j-1}] \subseteq\mathcal K^j\subseteq\mathcal H,$ yields
		$\mathcal H\subsetneq\mathcal H+\mathcal K^{j-1}
		\subseteq N_{\mathcal K}(\mathcal H)$.
	\end{proof}
	
	Let $\mathcal H$ be a nilpotent conformal subalgebra of $\mathcal L$, and regard $\mathcal L$ as a conformal $\mathcal H$-module under the adjoint action. In the notation of Definition~\ref{def2.3} with $\phi=0$ we have $\mathcal L^0=\bigcup\limits_{j\geq-1}\mathcal L_j^0.$
	
	\begin{proposition}\label{ca:zero-weight-criterion}
		The zero-weight space $\mathcal L^0$ is a conformal subalgebra of $\mathcal L$. Moreover, $\mathcal H$ is a Cartan subalgebra of $\mathcal L$ if and only if $\mathcal L^0=\mathcal H$.
	\end{proposition}
	
	\begin{proof}
		Nilpotence of $\mathcal H$ implies $\mathcal H\subseteq\mathcal L^0$. First, argue by induction on $i+j$ we prove that
		\begin{equation}\label{ca:zero-bracket-filtration}
			[\mathcal L_i^0\cdot\mathcal L_j^0]\subseteq\mathcal L_{i+j}^0,
			\qquad i,j\geq0.
		\end{equation}
		For $x\in\mathcal L_i^0$, $y\in\mathcal L_j^0$, and $c\in\mathcal H$, the Jacobi identity gives
		$$[c_\mu[x_\lambda y]]=[[c_\mu x]_{\lambda+\mu}y]
		+[x_\lambda[c_\mu y]].$$
		If $i=j=0$, both terms vanish. Otherwise, the induction hypothesis places the right-hand side in $\mathcal L_{i+j-1}^0[\lambda,\mu]$. Thus, $[x_\lambda y]$ lies in $\mathcal L_{i+j}^0$, proving \eqref{ca:zero-bracket-filtration}. Consequently, $\mathcal L^0$ is a conformal subalgebra.
		
		Suppose that $\mathcal H$ is Cartan. If $\mathcal L^0\neq\mathcal H$, choose the least $j\geq0$ such that $\mathcal L_j^0\not\subseteq\mathcal H$. Then
		\[
		[\mathcal H_\lambda\mathcal L_j^0]
		\subseteq\mathcal L_{j-1}^0[\lambda]
		\subseteq\mathcal H[\lambda].
		\]
		Thus $\mathcal L_j^0\subseteq N_{\mathcal L}(\mathcal H),$ contradicting $N_{\mathcal L}(\mathcal H)=\mathcal H$. 
		
		Conversely, suppose that $\mathcal L^0=\mathcal H$. Then the ascending filtration stabilizes in $\mathcal H$ at some finite stage, say $\mathcal H=\mathcal L_s^0$. For $x\in N_{\mathcal L}(\mathcal H)$, we have $[\mathcal H_\lambda x]\subseteq\mathcal H[\lambda]=\mathcal L_s^0[\lambda],$ so $x\in\mathcal L_{s+1}^0\subseteq\mathcal H$. Hence $N_{\mathcal L}(\mathcal H)=\mathcal H$, and the nilpotence of $\mathcal H$ completes the proof.
	\end{proof}
	
	The following proposition shows that, for every element $a$ of a finite Lie conformal algebra $\mathcal L$, the iterated adjoint kernel $\mathcal L_{0,a}$ is a conformal subalgebra, and provides a criterion for it to be a Cartan subalgebra.
	\begin{proposition}\label{ca:single-element-criterion}
		For every $a\in\mathcal L$, the subspace $\mathcal L_{0,a}=\bigcup\limits_{n\geq1}\ker(\operatorname{ad}_a)^{n}$ is a conformal subalgebra of $\mathcal L$. It is a Cartan subalgebra if and only if it is nilpotent and contains $a$.
	\end{proposition}
	
	\begin{proof} Set $\mathcal K=\mathcal L_{0,a}$. Let $x,y\in\mathcal K$, and choose $p,q\geq1$ such that $(ad_a)^p(x)=0$ and $(ad_a)^q(y)=0.$ Iterating the Jacobi  identity, we obtain
		$$(ad_a)^m([x_\lambda y])=\sum_{S\subseteq\{1,\ldots,m\}}
		\left[(a_{\nu_S}x)_{\lambda+\sum\limits_{i\in S}\nu_i}
		(a_{\nu_{S^c}}y)\right],$$
		where, for $S=\{i_1<\cdots<i_s\}$ we use notations $a_{\nu_S}:=a_{\nu_{i_s}}\cdots a_{\nu_{i_1}},$ $a_{\nu_\varnothing}:=\operatorname{id},$ and $S^c$ denotes the complement of $S$ in $\{1,\ldots,m\}$. 
		
		Taking $m=p+q-1$, each summand vanishes, since either $|S|\geq p$ or $|S^c|\geq q$. Hence, $[x_\lambda y]\in\mathcal K[\lambda]$ and we obtain that $\mathcal K$ is a conformal subalgebra.
		
		Next we observe that $a\in N_{\mathcal L}(\mathcal K)$. Indeed, if $x\in\mathcal K$ is annihilated by $p$ successive $a$-actions, then $(ad_a)^p([a_\mu x])=0,$ because every iterated $a$-action of length at least $p$ annihilates $x$. Comparing coefficients of $\mu$ shows that $[a_\mu x]\in\mathcal K[\mu]$. 
		Consequently, if $\mathcal K$ is Cartan, then
		$a\in N_{\mathcal L}(\mathcal K)=\mathcal K$, and $\mathcal K$ is nilpotent.
		
		Conversely, suppose that $a\in\mathcal K$ and $\mathcal K$ is nilpotent with nilindex $n+1$. For $x\in N_{\mathcal L}(\mathcal K)$, skew-symmetry gives $[a_\mu x]\in\mathcal K[\mu].$
		Since $a\in\mathcal K$, applying $n$ further $a$-actions yields
		$(ad_a)^n([a_\mu x])=0$ and hence, $x\in\mathcal K$. Thus, $N_{\mathcal L}(\mathcal K)\subseteq\mathcal K.$ The reverse inclusion follows from the fact that $\mathcal K$ is a conformal subalgebra. Consequently, $N_{\mathcal L}(\mathcal K)=\mathcal K.$ Together with the assumed nilpotence of $\mathcal K$, this proves that $\mathcal K$ is a Cartan subalgebra of $\mathcal L$.
	\end{proof}
	
	Let $\mathcal H$ be a Cartan subalgebra of $\mathcal L$, and regard
	$\mathcal L$ as an $\mathcal H$-module under the adjoint action.
	Fix a finite-dimensional subspace $E\subseteq\mathcal H$
	such that $\C[\partial]E=\mathcal H$, and define 
	$$\mathcal{U}_E =E\setminus \bigcup_{0\ne\phi\in\operatorname{W}(\mathcal L)} \ker(\phi_{|E}). $$
	
	\begin{proposition}
		The set $\mathcal{U}_E$ is nonempty, and $\mathcal L_{0,a}=\mathcal H$ for every $a\in\mathcal{U}_E$.
	\end{proposition}
	\begin{proof} Since $\C[\partial]E=\mathcal H$ and
		$\phi_{f(\partial)h}(\lambda)=f(-\lambda)\phi_h(\lambda)$,
		every nonzero weight $\phi$ has nonzero restriction to $E$.
		Thus each $\ker(\phi_{|E})$ is a proper subspace of $E$.
		The finiteness of $\operatorname{W}(\mathcal L)$ and the fact that $E$ cannot be covered by finitely many proper linear subspaces imply that $\mathcal{U}_E$ is nonempty.
		
		For $a\in\mathcal{U}_E$, we have $\phi_a=0$ if and only if $\phi=0$. Hence \eqref{eq4.3} and Proposition~\ref{ca:zero-weight-criterion} give $\mathcal L_{0,a}=\bigoplus\limits_{\phi_a=0}\mathcal L^\phi=\mathcal L^0=\mathcal H.$
	\end{proof}
	
	The following theorem, one of the main results of this paper, extends the Cartan existence result of D'Andrea and Marchei \cite{DM2012} to finite solvable {\it LCA}s.
	
	\begin{theorem}\label{ca:existence}
		Every finite solvable {\it LCA} admits a Cartan subalgebra.
	\end{theorem}
	
	\begin{proof}
		Suppose first that $\mathcal{R}$ is free of rank $r$. Let
		$$0=\mathcal{M}_0\subset \mathcal{M}_1\subset\cdots\subset \mathcal{M}_r=\mathcal{R}$$
		be an adjoint ideal flag as in \eqref{ca:triangular-character}, and choose a finite-dimensional subspace $E \subseteq \mathcal{R}$ such that $\C[\partial]E=\mathcal{R}$.
		For every nonzero $\phi_i$, ${\phi_i}_{|E}\ne0$ and $\ker({\phi_i}_{|E})$ is a proper linear subspace of $E$. Since $E$ cannot be covered by finitely many $\ker({\phi_i}_{|E})$, we may choose $a_0\in \mathcal{U}_E$ such that $\phi_i(a_0)\ne0$. 
		Let $\langle a_0\rangle$ denote the conformal subalgebra generated by $a_0$ of $\mathcal{U}_E$.
		
		Since $\mathcal{R}$ is finite over the Noetherian ring $\C[\partial]$, every submodule of $\mathcal{R}$ is finitely generated; hence $\langle a_0\rangle$ is finite. Moreover, it is solvable as a conformal subalgebra of the solvable algebra $\mathcal{R}$. By
		\cite[Corollary~4.1]{DM2012}, there exists $a\in\langle a_0\rangle$ such that $a-a_0\in\langle a_0\rangle^2$ and $\langle a\rangle$ is nilpotent. Since $\phi_i(\mathcal{R}^2)=0$, it gives that $\phi_i(a)=\phi_i(a_0) \neq 0$ for every $i$. Set $\mathcal H=\mathcal{R}_{0,a}$ and $\mathcal H_i=\mathcal H\cap \mathcal{M}_i, \ 0\leq i\leq r.$
		Since $\langle a\rangle$ is nilpotent, $a\in \mathcal H.$ By Proposition~\ref{ca:single-element-criterion}, it remains to prove that $\mathcal H$ is nilpotent.

		Let fix $i$, $1\leq i\leq r$. Then we distinguish two cases.
		
		Assume first that $\phi_i\ne0$. Let $x\in \mathcal H_i$. Write its image in
		$\mathcal{M}_i/\mathcal{M}_{i-1}=\C[\partial]\bar m_i$ as $\bar x=f(\partial)\bar m_i$ for some $f\in \C[\partial]$. Since $x\in \mathcal H=\mathcal{R}_{0,a}$, there exists $N$ such that $(\operatorname{ad}_{a})^{N}(x)=0$. 
		Using $a_\lambda\bar m_i=\phi_i(a)\bar m_i,$
		we obtain
		$$0=a_{\lambda_N}\cdots a_{\lambda_1}\bar x=\left(\prod_{k=1}^N\phi_i(a)(\lambda_k)\right)f(\partial+\lambda_1+\cdots+\lambda_N)\bar m_i.$$
		Since $\phi_i(a)\ne0$ and $\C[\partial][\lambda_1,\ldots,\lambda_N]$ is an integral domain and $\C[\partial]\bar m_i$ is free of rank one, the displayed equality implies $f=0.$ Thus $\bar x=0$, so $x\in \mathcal{M}_{i-1}$. Hence $\mathcal H_i\subseteq \mathcal H_{i-1}.$ The reverse inclusion is immediate, and therefore $\mathcal H_i=\mathcal H_{i-1}$ for $\phi_i\ne0.$ Taking into account that 
		$\mathcal{M}_{i-1}$ is an ideal, we deduce
		$$[\mathcal H\cdot \mathcal H_i]=[\mathcal H\cdot \mathcal H_{i-1}]\subseteq \mathcal H\cap \mathcal{M}_{i-1}=\mathcal H_{i-1}.$$
		
		Assume now that $\phi_i=0$. Since the action of $\mathcal{R}$ on
		$\mathcal{M}_i/\mathcal{M}_{i-1}$ is given by the character $\phi_i$ and by Proposition \ref{ca:single-element-criterion}, we have $[\mathcal{R}\cdot \mathcal{M}_i]\subseteq \mathcal{M}_{i-1}$ and $[\mathcal H\cdot \mathcal H_i]\subseteq \mathcal H.$ Therefore
		$$[\mathcal H\cdot \mathcal H_i]\subseteq[\mathcal{R}\cdot \mathcal{M}_i]\subseteq \mathcal{M}_{i-1}, \quad 
		[\mathcal H\cdot \mathcal H_i]\subseteq \mathcal H\cap \mathcal{M}_{i-1}=\mathcal H_{i-1}.$$
		Thus, $[\mathcal H\cdot \mathcal H_i]\subseteq \mathcal H_{i-1},$ $1\leq i\leq r$. Since $\mathcal H=\mathcal H_r$, repeated application of the last embedding gives $\mathcal H^{r+1}\subseteq \mathcal H_0=0.$
		Therefore $\mathcal H$ is nilpotent. 
		
		We now remove the freeness assumption. Let $\mathcal{T}=\operatorname{Tor}(\mathcal{R})$
		and let $\pi:\mathcal{R}\rightarrow \bar{\mathcal{R}}=\mathcal{R}/\mathcal{T}$ 
		be the quotient map. Then $\bar{\mathcal{R}}$ is free and solvable. Therefore, by the free case, $\bar{\mathcal{R}}$ admits a Cartan subalgebra $\bar{\mathcal H}$ (say with nilindex $s$). Setting $\mathcal H=\pi^{-1}(\bar{\mathcal H})$ we have 
		$\pi(\mathcal H^{s})\subseteq \bar{\mathcal H}^{\,s}= \bar 0,$ which leads $\mathcal H^{s}\subseteq\mathcal{T}.$ Since $\mathcal{T}\subseteq \operatorname{Center}(\mathcal{R})$,
		$\mathcal H^{s+1}=[\mathcal H^{s}\cdot \mathcal H]=0.$ Thus, $\mathcal H$ is nilpotent.
		
		Let $x\in N_{\mathcal{R}}(\mathcal H).$ Then
		$[\pi(x)_\lambda\bar{\mathcal H}]\subseteq\bar{\mathcal H}[\lambda],$ so $\pi(x)\in N_{\bar{\mathcal{R}}}(\bar{\mathcal H})=\bar{\mathcal H}.$ Hence
		$x\in\pi^{-1}(\bar{\mathcal H})=\mathcal H.$
		Thus, $N_{\mathcal{R}}(\mathcal H)\subseteq \mathcal H.$
		The reverse inclusion holds because $\mathcal H$ is a conformal
		subalgebra. Thus $\mathcal H$ is self-normalizing and therefore
		is a Cartan subalgebra of $\mathcal R$.
	\end{proof}
	
	It should be noted that the restriction to solvable algebras in Theorem \ref{ca:existence} is essential. Indeed, in the case of {\it LCA} is $\operatorname{Vir}$, there is no nontrivial  saturated submodules and consequently, this algebra does not admits a Cartan subalgebra.
	
	The exactness of generalized zero-weight submodules also gives the following behavior of Cartan subalgebras under quotients.
	
	Let $\varphi : \mathcal L_1\to\mathcal L_2$ be a surjective homomorphism of {\it LCA}s and let $\mathcal H$ be a
	Cartan subalgebra of $\mathcal L_1$. Set
	$\mathcal I=\ker \varphi$ and $\widetilde{\mathcal H}=\varphi(\mathcal H)$.
	
	\begin{proposition}\label{ca:cartan-quotient}
		The subalgebra $\widetilde{\mathcal H}$ is a Cartan subalgebra
		of $\mathcal L_2$. Moreover,
		$$\varphi^{-1}\bigl(N_{\mathcal L_2}(\widetilde{\mathcal H})\bigr) = N_{\mathcal L_1}(\mathcal H+\mathcal I)
		= \mathcal H+\mathcal I.$$
	\end{proposition}
	\begin{proof} Regard $\mathcal L_2$ as a conformal $\mathcal H$-module via the action $h_\lambda x=[\varphi(h)_\lambda x],$ $h\in\mathcal H,\ x\in\mathcal L_2.$
		The applications of Proposition \ref{ca:weight-exactness} at weight zero and Proposition~\ref{ca:zero-weight-criterion} to the short exact sequence of conformal $\mathcal H$-modules
		\[
		0\longrightarrow\mathcal I
		\longrightarrow\mathcal L_1
		\xrightarrow{\varphi}\mathcal L_2
		\longrightarrow 0
		\]
		imply 
		\[ \mathcal L_2^0 = \varphi(\mathcal L_1^0) = \varphi(\mathcal H)
		= \widetilde{\mathcal H}.
		\]
		
		Since $\widetilde{\mathcal H}$ is nilpotent as a homomorphic image of $\mathcal H$, by Proposition~\ref{ca:zero-weight-criterion} we obtain that
		$\widetilde{\mathcal H}$ is a Cartan subalgebra of $\mathcal L_2$.
		
		Since $\varphi^{-1}(\widetilde{\mathcal H})=\mathcal H+\mathcal I$
		and $\mathcal I$ is an ideal of $\mathcal L_1$, for every
		$x\in\mathcal L_1$ we have
		\begin{align*}
			\varphi(x)\in N_{\mathcal L_2}(\widetilde{\mathcal H})
			&\Longleftrightarrow
			[x_\lambda\mathcal H]
			\subseteq(\mathcal H+\mathcal I)[\lambda]\\
			&\Longleftrightarrow
			[x_\lambda(\mathcal H+\mathcal I)]
			\subseteq(\mathcal H+\mathcal I)[\lambda]\\
			&\Longleftrightarrow
			x\in N_{\mathcal L_1}(\mathcal H+\mathcal I).
		\end{align*}
		Consequently, $\varphi^{-1}\bigl(N_{\mathcal L_2}(\widetilde{\mathcal H})\bigr) = N_{\mathcal L_1}(\mathcal H+\mathcal I).$ Finally, the self-normalizing property of
		$\widetilde{\mathcal H}$ gives
		$N_{\mathcal L_1}(\mathcal H+\mathcal I) = \varphi^{-1}(\widetilde{\mathcal H}) = \mathcal H+\mathcal I$
		as required.
	\end{proof}
	
	For central quotients, Proposition \ref{ca:cartan-quotient}  gives a particularly useful correspondence. Let $\mathcal I\subseteq\operatorname{Center}(\mathcal L)$ be an ideal, and let $\varphi : \mathcal L\longrightarrow\mathcal L/\mathcal I$ be the quotient map.
	
	\begin{corollary}\label{ca:central-quotient}
		The map $\mathcal H\mapsto\mathcal H/\mathcal I$
		defines a bijection between the Cartan subalgebras of $\mathcal L$ and those of $\mathcal L/\mathcal I$. 
	\end{corollary}
	
	\begin{proof}
		By Lemma~\ref{ca:normalizer-facts}, every Cartan subalgebra of $\mathcal L$
		contains $\mathcal I$. Hence Proposition~\ref{ca:cartan-quotient} shows that
		its image in $\mathcal L/\mathcal I$ is Cartan. Conversely, let $\overline{\mathcal H}$ be a Cartan subalgebra of
		$\mathcal L/\mathcal I$, and set
		$\mathcal H=\varphi^{-1}(\overline{\mathcal H}).$
		If $\overline{\mathcal H}^{\,s}=0$, then
		$\mathcal H^{s+1}=[\mathcal H^{s}\cdot\mathcal H] \subseteq
		[\mathcal H\cdot\mathcal I]= 0.$
		Thus $\mathcal H$ is nilpotent.
		
		If $x\in N_{\mathcal L}(\mathcal H)$, then
		$\varphi(x)\in N_{\mathcal L/\mathcal I}(\overline{\mathcal H})
		= \overline{\mathcal H}.$ Hence, 
		$x\in \varphi^{-1}(\overline{\mathcal H})=\mathcal H.$
		Therefore, $N_{\mathcal L}(\mathcal H)=\mathcal H$
		and $\mathcal H$ is Cartan.
	\end{proof}
	
	Applying Proposition~\ref{ca:cartan-quotient} to the abelianization
	gives the following decomposition.
	
	\begin{corollary}\label{ca:cartan-derived}
		Let $\mathcal H$ be a Cartan subalgebra of $\mathcal{L}$. Then
		$\mathcal{L}=\mathcal H+\mathcal{L}^2$ and there is an exact sequence
		$$0\longrightarrow \mathcal H\cap \mathcal{L}^2\longrightarrow \mathcal H\longrightarrow \mathcal{L}/\mathcal{L}^2\longrightarrow 0.$$
	\end{corollary}
	\begin{proof}
		Applying Proposition~\ref{ca:cartan-quotient} to the natural quotient map $\varphi:\mathcal{L} \rightarrow \mathcal{L}/\mathcal{L}^2,$ we obtain that $\varphi(\mathcal H)=\mathcal{L}/\mathcal{L}^2$ is the unique Cartan subalgebra of $\mathcal{L}/\mathcal{L}^2$. Since, $\varphi(\mathcal H)=(\mathcal H+\mathcal{L}^2)/\mathcal{L}^2,$ we obtain  
		$\mathcal{L}= \mathcal H+\mathcal{L}^2.$
		
		The restriction of $\varphi$ to $\mathcal H$ is therefore surjective with kernel $\ker \varphi_{|\mathcal H}=\mathcal H\cap \mathcal{L}^2.$
		Consequently, the sequence 
		$$0\longrightarrow \mathcal H\cap \mathcal{L}^2\longrightarrow \mathcal H\xrightarrow{\,\varphi_{|\mathcal H}\,}\mathcal{L}/\mathcal{L}^2\longrightarrow 0$$
		is exact.
	\end{proof}
	
	For free solvable {\it LCA}s, the triangular flag also
	determines the rank of a Cartan subalgebra.
	
	\begin{proposition}\label{ca:cartan-rank}
		Let $\mathcal{R}$ be free of rank $r$, and let
		$$0=\mathcal{M}_0\subset \mathcal{M}_1\subset\cdots\subset \mathcal{M}_r=\mathcal{R}$$
		be an adjoint ideal flag as in \eqref{ca:triangular-character}, with corresponding characters $\phi_i$. If $\mathcal H$ is a Cartan subalgebra of $\mathcal{R}$, then
		$\operatorname{rank}(\mathcal H)=\operatorname{card}\{\,i\mid\phi_i=0\,\}.$
	\end{proposition}
	\begin{proof} Regard $\mathcal R$ as a conformal $\mathcal H$-module under the adjoint action, and set $\mathcal H_i=\mathcal H\cap\mathcal M_i,$ $0\leq i\leq r.$ Since $\mathcal H=\mathcal R^0$, we have
		$\mathcal M_i^0=\mathcal H_i$.
		Applying \eqref{ca:weight-exact-equation} to the short exact
		sequence of conformal $\mathcal H$-modules
		\[
		0\longrightarrow\mathcal M_{i-1}
		\longrightarrow\mathcal M_i
		\longrightarrow\mathcal M_i/\mathcal M_{i-1}
		\longrightarrow0
		\]
		therefore yields
		\[
		0\longrightarrow\mathcal H_{i-1}
		\longrightarrow\mathcal H_i
		\longrightarrow(\mathcal M_i/\mathcal M_{i-1})^0
		\longrightarrow0,
		\]
		where the superscript $0$ denotes the generalized zero-weight
		submodule for the restricted action of $\mathcal H$.
		
		The quotient $\mathcal M_i/\mathcal M_{i-1}
		=\C[\partial]\bar m_i$ carries the action $a_\lambda\bar m_i=\phi_i(a)(\lambda)\bar m_i,$ $a\in\mathcal H.$
		Consequently,
		\[
		(\mathcal M_i/\mathcal M_{i-1})^0
		=
		\begin{cases}
			\C[\partial]\bar m_i, & {\phi_i}_{| \mathcal H}=0,\\
			0, & {\phi_i}_{| \mathcal H}\ne0.
		\end{cases}
		\]
		Taking ranks in the preceding exact sequence gives
		\[
		\operatorname{rank}\mathcal H_i
		-\operatorname{rank}\mathcal H_{i-1}
		=\begin{cases}
			1, & {\phi_i}_{ | \mathcal H}=0,\\
			0, & {\phi_i}_{| \mathcal H}\ne 0.
		\end{cases}
		\]
		Since $\mathcal H_0=0$ and $\mathcal H_r=\mathcal H$,
		summing over $i$ yields $\operatorname{rank}\mathcal H
		=\operatorname{card}\{ \ i \ | \ {\phi_i}_{| \mathcal H}=0\,\}.$
		
		Finally, Corollary~\ref{ca:cartan-derived} gives
		$\mathcal R=\mathcal H+\mathcal R^2$, while
		\eqref{ca:triangular-character} implies
		$\phi_i(\mathcal R^2)=0$ for every $i$.
		Thus ${\phi_i}_{| \mathcal H}=0$ if and only if $\phi_i=0$,
		and hence $\operatorname{rank}\mathcal H = \operatorname{card}
		\{ \ i \ | \ \phi_i=0\,\}.$
	\end{proof}
	
	Recall, that for a finite-dimensional Lie algebra  $\mathfrak g$ over $\C$ the associated
	\emph{current LCA} is
	$\operatorname{Cur}(\mathfrak g)=\C[\partial]\otimes\mathfrak g,$
	with $\lambda$-bracket determined by
	$[x_\lambda y]=[x,y]_{\mathfrak g},~ x,y\in\mathfrak g,$
	and extended to $\operatorname{Cur}(\mathfrak g)$ by sesquilinearity.
	
	For a submodule $\mathcal M\subseteq\operatorname{Cur}(\mathfrak g)$ we set $\operatorname{coeff}(\mathcal M)=\operatorname{Span}_{\mathbb{C}}\{ \ x_i \ | \ x=\sum x_i\partial^i \in \mathcal{M} \}.$ The current bracket and sesquilinearity give  
	\begin{equation}\label{eq4.5}
		\operatorname{coeff}([\mathcal U\cdot\mathcal V])
		=[\operatorname{coeff}(\mathcal U),\operatorname{coeff}(\mathcal V)]_{\mathfrak g}   
	\end{equation}
	for any submodules $\mathcal U,\mathcal V\subseteq\operatorname{Cur}(\mathfrak g).$
	
	The existence statement of Theorem~\ref{ca:existence} does not
	extend to finite semisimple {\it LCA}s. By \cite[Theorem~7.1]{DK},
	these are finite direct sums of
	\[
	\operatorname{Vir},\qquad
	\operatorname{Cur}(\mathfrak s),\qquad
	\operatorname{Vir}\ltimes\operatorname{Cur}(\mathfrak g),
	\]
	where $\mathfrak s$ is finite-dimensional and simple, and
	$\mathfrak g$ is finite-dimensional and semisimple.
	Since $\operatorname{Vir}$ has no nonzero nilpotent conformal
	subalgebra, it admits no Cartan subalgebra.

	The standard semidirect product
	$\mathcal L=\operatorname{Vir}\ltimes\operatorname{Cur}(\mathfrak g)$
	also admits no Cartan subalgebra. Indeed, if $\mathcal H$ were
	such a subalgebra, its image under the projection onto
	$\operatorname{Vir}$ would be nilpotent and hence zero. Thus
	$\mathcal H\subseteq\operatorname{Cur}(\mathfrak g)$.
	By \eqref{eq4.5}, the Lie algebra
	$\mathfrak h=\operatorname{coeff}(\mathcal H)$ is nilpotent,
	and so is $\operatorname{Cur}(\mathfrak h)$.
	Since $\mathcal H\subseteq\operatorname{Cur}(\mathfrak h)$,
	self-normalization and Lemma~\ref{ca:normalizer-facts} imply
	$\mathcal H=\operatorname{Cur}(\mathfrak h)$. However, the standard Virasoro action leads to the contradiction to the self-normalizing property of $\mathcal H$.
	
	Among the semisimple building blocks considered above, the current {\it LCA}s are precisely those admitting Cartan subalgebras. Below we establish that their Cartan theory is completely determined by the classical Cartan theory of the underlying finite-dimensional Lie algebra \cite{Humphreys}.
	
	\begin{proposition}\label{ca:current-correspondence}
		Let $\mathfrak g$ be a finite-dimensional complex Lie algebra. A conformal subalgebra $\mathcal H$ of $\operatorname{Cur}(\mathfrak g)$ is Cartan if and only if $\mathcal H=\operatorname{Cur}(\mathfrak h)$ for some Cartan subalgebra $\mathfrak h$ of $\mathfrak g$.
	\end{proposition}
	\begin{proof}
		Suppose first that $\mathcal H$ is a Cartan subalgebra of
		$\operatorname{Cur}(\mathfrak g)$, and set
		$\mathfrak h=\operatorname{coeff}(\mathcal H)$.
		By \eqref{eq4.5}, $\mathfrak h$ is a nilpotent. For $x=\sum\limits_i\partial^i x_i\in\operatorname{Cur}(\mathfrak g)$ we have that $[x_\lambda\operatorname{Cur}(\mathfrak h)]
		\subseteq\operatorname{Cur}(\mathfrak h)[\lambda]$ if and only if $[x_i,\mathfrak h]_{\mathfrak g}\subseteq\mathfrak h$
		for every $i$, or equivalently,
		$x_i\in N_{\mathfrak g}(\mathfrak h)$ for every $i$. Hence, 
		\begin{equation}\label{eq4.6}
			N_{\operatorname{Cur}(\mathfrak g)}
			\bigl(\operatorname{Cur}(\mathfrak h)\bigr)
			= \operatorname{Cur}\bigl(N_{\mathfrak g}(\mathfrak h)\bigr).
		\end{equation}
		Since $\operatorname{Cur}(\mathfrak h)$ is nilpotent and contains
		$\mathcal H$, self-normalization and Lemma~\ref{ca:normalizer-facts} imply $\mathcal H=\operatorname{Cur}(\mathfrak h)$.
		Equation~\eqref{eq4.6} therefore gives $\operatorname{Cur}(\mathfrak h) = \operatorname{Cur}\bigl(N_{\mathfrak g}(\mathfrak h)\bigr),$ whence $N_{\mathfrak g}(\mathfrak h)=\mathfrak h$.
		Thus, $\mathfrak h$ is a Cartan subalgebra of $\mathfrak g$.
		
		Conversely, let $\mathfrak h$ be a Cartan subalgebra of
		$\mathfrak g$. Then $\operatorname{Cur}(\mathfrak h)$ is
		nilpotent. Since \eqref{eq4.6} holds
		for every Lie subalgebra $\mathfrak h\subseteq\mathfrak g$, we conclude that $\operatorname{Cur}(\mathfrak h)$ is self-normalizing
		in $\operatorname{Cur}(\mathfrak g)$. Hence, it is a Cartan
		subalgebra.
	\end{proof}
	
	\subsection{Conjugacy of Cartan subalgebras.}
	
	For a Lie conformal algebra $\mathcal L$, its automorphism group is
	$$\operatorname{Aut}(\mathcal L) = \left\{
	\varphi\in\operatorname{Aut}_{\C[\partial]}(\mathcal L)
	\;\middle|\; \varphi([x_\lambda y]) = [\varphi(x)_\lambda\varphi(y)] \text{ for all }x,y\in\mathcal L
	\right\},$$
	where $\operatorname{Aut}_{\C[\partial]}(\mathcal L)$ denotes
	the group of automorphisms of the underlying $\C[\partial]$-module.
	
	For $x\in\mathcal L$, let $d_x:=(\operatorname{ad}_x)_0\in\operatorname{Der}(\mathcal L)$. If $d_x$ is nilpotent, then $\exp(d_x)=\sum_{k\geq0}\frac{d_x^k}{k!}\in\operatorname{Aut}(\mathcal L).$ An automorphism of this form is called an \emph{inner automorphism}. We denote by $\operatorname{InAut}(\mathcal L)$ the subgroup of $\operatorname{Aut}(\mathcal L)$ generated by all such automorphisms.
	
	By \cite[Corollary~8.3]{DK}, the solvability of $\mathcal R$ implies the nilpotence of $\mathcal R^2$. Hence, for every $x\in\mathcal R^2$, the zero mode $d_x=(\operatorname{ad}_x)_0$ is a nilpotent derivation of $\mathcal R$. Therefore, $\exp(d_x)\in\operatorname{InAut}(\mathcal R).$
	
	\begin{lemma}\label{ca:graph}
		Let $\mathcal L=\mathcal H\ltimes\mathcal A$ be a Lie conformal
		algebra, where $\mathcal H$ is nilpotent and
		$\mathcal A=\mathbb C[\partial]a$ is an abelian ideal with the action on $\mathcal{H}$ via $[h_\lambda a]=\phi_h(\lambda)a,$ $h\in\mathcal H,$ for a nonzero weight $\phi$. Then every conformal subalgebra
		$\mathcal K$ complementary to $\mathcal A$ in $\mathcal L$ has the
		form
		\begin{equation}\label{ca:all-graphs}
			\mathcal K
			=
			\bigl\{h+c\phi_h(-\partial)a \,\bigm|\, h\in\mathcal H\bigr\}
			=
			\exp(-c d_a)(\mathcal H)
		\end{equation}
		for a unique $c\in\mathbb C$.
	\end{lemma}
	
	\begin{proof}
		The projection $\pi:\mathcal L\to\mathcal H$ restricts to an
		isomorphism of modules $\pi|_{\mathcal K}:\mathcal K\to\mathcal H$. Hence there exists a unique $\C[\partial]$-linear map $F:\mathcal H\rightarrow\mathcal A$ such that
		$\mathcal K=\{\,h+F(h)\mid h\in\mathcal H\,\}.$  Write $F(h)=q_h(\partial)a$, with $q_h\in\mathbb C[\partial]$.
		Since $\mathcal A$ is abelian, the condition of $\mathcal K$ being subalgebra is equivalent to
		$$F([h_1{}_\lambda h_2]) = [h_1{}_\lambda F(h_2)] + [F(h_1){}_\lambda h_2]. $$
		By sesquilinearity and skew-symmetry, this identity becomes
		\begin{equation}\label{ca:graph-equation}
			F([h_1{}_\lambda h_2]) = \bigl(
			q_{h_2}(\partial+\lambda)\phi_{h_1}(\lambda) - q_{h_1}(-\lambda)\phi_{h_2}(-\partial-\lambda)\bigr)a.
		\end{equation}
		The Jacobi identity leads $\phi(\mathcal H^2)=0$. Choose $\tilde h\in\mathcal H$ such that $p(\lambda):=\phi_{\tilde h}(\lambda)\ne0$.
		
		Let $s$ is the nilindex of $\mathcal{H}$. We prove by descending induction on $j$ that 
		$F(\mathcal H^j)=0,$ $2\leq j\leq s.$ The case $j=s$ is immediate. Suppose that $F(\mathcal H^{j+1})=0$ for some $j\geq2$, and let $h_1\in\mathcal H^j\subseteq\mathcal H^2$, we have $\phi_{h_1}=0$, while $[\tilde h_\lambda h_1]\in\mathcal H^{j+1}[\lambda].$ Hence $F([\tilde h_\lambda h_1])=0$, and \eqref{ca:graph-equation} gives $q_{h_1}(\partial+\lambda)p(\lambda)=0.$ Since $p\ne0$, it follows that $q_{h_1}=0$, and hence $F(h_1)=0$. Thus $F(\mathcal H^j)=0$, completing the induction. In particular, $F(\mathcal H^2)=0.$ Therefore, for arbitrary $h\in\mathcal H$, we have
		$F([\tilde h_\lambda h])=0$.
		Equation~\eqref{ca:graph-equation} consequently gives
		$$q_h(\partial+\lambda)p(\lambda) = q_{\tilde h}(-\lambda)\phi_h(-\partial-\lambda).$$
		Then there exists $c\in \mathbb{C}$ such that $q_h(\partial)=c\phi_h(-\partial),$ $h\in\mathcal H.$
		Thus, we obtain $\mathcal K=\bigl\{h+c\phi_h(-\partial)a\mid h\in\mathcal H\bigr\}.$
		
		Taking into account that  $d_a(h)=-\phi_h(-\partial)a,$ $h\in\mathcal H$ and $d_a(\mathcal A)=0,$ 
		we obtain $d_a^2=0.$ Hence, $\exp(-c d_a)=\operatorname{id}-c d_a$
		is an automorphism of $\mathcal L$ satisfying $\exp(-c d_a)(h)=h+c\phi_h(-\partial)a.$ This proves \eqref{ca:all-graphs}.
		The uniqueness of $c$ follows from the uniqueness of the 
		map $F$.
	\end{proof}
	
	We now prove that any two Cartan subalgebras of $\mathcal R$ are conjugate under $\operatorname{InAut}(\mathcal R)$.
	\begin{theorem}\label{ca:conjugacy}
		Any two Cartan subalgebras of $\mathcal R$ are conjugate by a finite
		product of exponentials of nilpotent zero modes of elements of
		$\mathcal R^2$.
	\end{theorem}
	
	\begin{proof}
		Let $\mathcal H_1$ and $\mathcal H_2$ be Cartan subalgebras of
		$\mathcal R$. We first consider the case of $\mathcal R$ is free,
		arguing by induction on $r=\operatorname{rank}\mathcal R$. For $r\leq 1$, solvability implies that $\mathcal R$ is abelian,
		and the assertion is immediate.
		
		Suppose that $r>1$. Choose an adjoint ideal flag as in \eqref{ca:triangular-character}, and set
		$\mathcal A=\mathcal M_1=\C[\partial]m$.
		Then $\mathcal A$ is a saturated abelian ideal with the products $[a_\lambda m]=\phi_a(\lambda)m,$ $a\in\mathcal R.$ Note that 
		$d_x$ is nilpotent for every $x\in\mathcal R^2$.
		
		Let $\pi : \mathcal R\rightarrow \overline{\mathcal R}:=\mathcal R/\mathcal A$ be the quotient map. By Proposition~\ref{ca:cartan-quotient},
		the subalgebras $\pi(\mathcal H_1)$ and $\pi(\mathcal H_2)$
		are Cartan subalgebras of $\overline{\mathcal R}$.
		Since $\mathcal A$ is saturated, $\overline{\mathcal R}$ is free
		of rank $r-1$. The induction hypothesis therefore gives
		$\bar x_1,\ldots,\bar x_s\in\overline{\mathcal R}^{\,2}$ such that
		\[
		\pi(\mathcal H_2)
		=
		\exp(d_{\bar x_s})\cdots
		\exp(d_{\bar x_1})\bigl(\pi(\mathcal H_1)\bigr).
		\]
		Using $\pi(\mathcal R^2)=\overline{\mathcal R}^{\,2}$, choose
		$x_j\in\mathcal R^2$ with $\pi(x_j)=\bar x_j$.
		The identity $\pi\circ d_{x_j}=d_{\bar x_j}\circ\pi$ and the nilpotence of these zero modes imply $\pi\circ\exp(d_{x_j})
		=\exp(d_{\bar x_j})\circ\pi.$
		
		Thus, setting
		$$\Phi=\exp(d_{x_s})\cdots\exp(d_{x_1}),
		\qquad \mathcal H_3=\Phi(\mathcal H_1),$$
		we obtain
		\begin{equation}\label{ca:common-preimage}
			\mathcal H_3+\mathcal A = \mathcal H_2+\mathcal A =:\mathcal K.
		\end{equation}
		By Lemma~\ref{ca:normalizer-facts}, each $\mathcal H_j$,
		$j=2,3$, is saturated in $\mathcal R$.
		It follows that $\mathcal A\cap\mathcal H_j$ is saturated
		in the rank-one free module $\mathcal A$, and hence $\mathcal A\subseteq\mathcal H_j$ or $\mathcal A\cap\mathcal H_j=0.$
		
		If $\mathcal A\subseteq\mathcal H_2$, then
		\eqref{ca:common-preimage} gives
		$\mathcal H_3\subseteq\mathcal H_2$.
		Lemma~\ref{ca:normalizer-facts} yields
		$\mathcal H_3=\mathcal H_2$.
		The same argument applies if
		$\mathcal A\subseteq\mathcal H_3$. We may therefore assume that $\mathcal A\cap\mathcal H_2 =
		\mathcal A\cap\mathcal H_3 =0.$ 
		Equation~\eqref{ca:common-preimage} then gives
		\[
		\mathcal K
		=
		\mathcal H_j\oplus_{\C[\partial]}\mathcal A
		=
		\mathcal H_j\ltimes\mathcal A,
		\qquad j=2,3.
		\]
		Moreover, $\phi|_{\mathcal H_j}\neq0$ for $j=2,3$.
		Indeed, if this restriction vanished, $\mathcal A$ would
		centralize $\mathcal H_j$ and would therefore lie in its
		normalizer. Since $\mathcal H_j$ is self-normalizing, this
		would imply $\mathcal A\subseteq\mathcal H_j$, a contradiction.
		
		Lemma~\ref{ca:graph} now provides a scalar $c\in\C$ such that $
		\mathcal H_2=\exp(-c\,d_m)(\mathcal H_3).$ 
		It remains to show that $m\in\mathcal R^2$.
		Choose $h\in\mathcal H_3$ with $\phi_h(\lambda)\neq0$.
		Then the product $[h_\lambda m]=\phi_h(\lambda)m$
		implies $m\in\mathcal R^2$. Consequently, 
		$\mathcal H_2 = \exp(d_{-cm})\Phi(\mathcal H_1),$ which proves the assertion in the free module case.
		
		For the general case, set
		$\mathcal T=\operatorname{Tor}(\mathcal R)$ and let $\pi:\mathcal R\longrightarrow
		\overline{\mathcal R}:=\mathcal R/\mathcal T$ be the quotient map.
		By Lemma~\ref{ca:normalizer-facts}, $\mathcal T\subseteq\operatorname{Center}(\mathcal R)
		\subseteq\mathcal H_1\cap\mathcal H_2.$
		The algebra $\overline{\mathcal R}$ is free,
		and Proposition~\ref{ca:cartan-quotient} shows that
		$\pi(\mathcal H_1)$ and $\pi(\mathcal H_2)$ are Cartan
		subalgebras of $\overline{\mathcal R}$.
		The free case therefore gives
		$\bar x_1,\ldots,\bar x_s\in\overline{\mathcal R}^{\,2}$
		with nilpotent zero modes such that
		$$\pi(\mathcal H_2) = \exp(d_{\bar x_s})\cdots
		\exp(d_{\bar x_1})\bigl(\pi(\mathcal H_1)\bigr).
		$$
		Choose lifts $x_j\in\mathcal R^2$ of $\bar x_j$.
		As before, one can obtain 
		$\pi\circ d_{x_j}=d_{\bar x_j}\circ\pi.$ The nilpotence of $d_{\bar x_j}$ leads 
		$d_{x_j}^{n_j}(\mathcal R)\subseteq\mathcal T \subseteq \operatorname{Center}(\mathcal{R})$ for some $n_j$. Therefore $d_{x_j}$ are nilpotent as well. Hence, we get $\pi\circ\exp(d_{x_j}) = \exp(d_{\bar x_j})\circ\pi.$
		It follows that $\Psi:=\exp(d_{x_s})\cdots\exp(d_{x_1})$ satisfies
		$\pi(\Psi(\mathcal H_1))=\pi(\mathcal H_2)$
		and $\Psi_{\mathcal T}=\operatorname{id}_{\mathcal T}$. 
		Both $\Psi(\mathcal H_1)$ and $\mathcal H_2$ contain
		$\mathcal T$, so
		$$\Psi(\mathcal H_1) = \pi^{-1}\bigl(\pi(\Psi(\mathcal H_1))\bigr)
		= \pi^{-1}\bigl(\pi(\mathcal H_2)\bigr) = \mathcal H_2.$$
		This completes the proof.
	\end{proof}
	
	The following result extends the inner conjugacy theorem for Cartan subalgebras of solvable {\it LCA}s to current Lie conformal algebras.
	
	\begin{proposition}\label{ca:current-conjugacy}
		Let $\mathfrak g$ be a finite-dimensional complex Lie algebra.
		Any two Cartan subalgebras of $\operatorname{Cur}(\mathfrak g)$ are conjugated via a finite product of automorphisms of the form 
		$\operatorname{id}_{\mathbb C[\partial]}\otimes\sigma_i$, where 
		$\sigma_i\in \operatorname{InAut}(\mathfrak g)$.
	\end{proposition}
	\begin{proof} Let $\mathcal H_1$ and $\mathcal H_2$ be Cartan subalgebras of $\operatorname{Cur}(\mathfrak g)$. By Proposition~\ref{ca:current-correspondence}, there exist
		Cartan subalgebras $\mathfrak h_1$ and $\mathfrak h_2$ of
		$\mathfrak g$ such that
		\[
		\mathcal H_i
		=
		\operatorname{Cur}(\mathfrak h_i)
		=
		\mathbb C[\partial]\otimes_{\mathbb C}\mathfrak h_i,
		\qquad i=1,2.
		\]
		The classical conjugacy theorem for Cartan subalgebras
		\cite{Bourbaki2005} provides an inner automorphism $\sigma$
		of $\mathfrak g$ satisfying $\sigma(\mathfrak h_1)=\mathfrak h_2$, where $\sigma=\prod\limits_{j=1}^{s}\exp(d_{x_j})$ for $x_j\in\mathfrak g$.
		Consequently,
		$$\bigl(\operatorname{id}_{\mathbb C[\partial]}\otimes\sigma\bigr)(\mathcal H_1)
		=\prod\limits_{j=1}^{s}\exp(d_{\mathbbm{1}\otimes x_j})(\mathcal H_1)
		=\mathbb C[\partial]\otimes_{\mathbb C}\prod\limits_{j=1}^{s}\exp({\operatorname{ad}}_{ x_j})(\mathfrak h_1)
		=\mathbb C[\partial]\otimes_{\mathbb C}\mathfrak h_2
		=\mathcal H_2,$$
		as required.
	\end{proof}
	
	Recall, that for a finite vertex algebra $\mathcal V$ the underlying $\C[\partial]$-module of $\mathcal V$, equipped with the $\lambda$-bracket $[a_\lambda b]=\sum\limits_{n\geq0}\frac{\lambda^n}{n!}\,a_{(n)}b,$
	is an {\it LCA}, denoted by $\mathcal V^{\mathrm{Lie}}.$ Moreover, for each $x\in\mathcal V$, the corresponding vertex operator is written as $Y(x,z)=\sum\limits_{n\in\mathbb Z}x_{(n)}z^{-n-1},$ where each $x_{(n)}\in\operatorname{End}_{\C}(\mathcal V)$. In particular, by the definition of the $\lambda$-bracket of $\mathcal V^{\mathrm{Lie}}$,
	we have $[x_\lambda a]_{\mid\lambda=0}=x_{(0)}a.$
	Therefore, we can identify the operator $x_{(0)}$ as $(\operatorname{ad}_x)_0.$ In the notation used above for {\it LCA}s, this is exactly the derivation $d_x$.
	
	In \cite[Remark~5.1]{DM2012}, D'Andrea and Marchei proved that Cartan subalgebras of a finite vertex algebra are conjugate under vertex algebra automorphisms. They further asked whether the conjugating automorphisms can always be chosen to be inner. Application of Theorem~\ref{ca:conjugacy} in the following result yields an affirmative answer.
	
	\begin{theorem}\label{ca:vertex-inner}
		Let $\mathcal V$ be a finite vertex algebra. Any two Cartan
		subalgebras of $\mathcal V$ are conjugate by a 
		vertex algebra automorphisms of the form
		$\sigma=\prod\limits_{j=1}^{s}\exp (d_{x_j}),$ $x_j\in(\mathcal V^{\mathrm{Lie}})^2.$
	\end{theorem}
	\begin{proof} Let $\mathcal U_1$ and $\mathcal U_2$ be Cartan subalgebras of $\mathcal V$, and set $\mathcal R=\mathcal V^{\mathrm{Lie}}$. By \cite[Theorem~3.3]{AD2008}, the Lie conformal algebra $\mathcal R$ is solvable. Moreover, by
		\cite[Remark~5.1]{DM2012}, the underlying conformal subalgebras
		of $\mathcal U_1$ and $\mathcal U_2$ are Cartan subalgebras of
		$\mathcal R$. Theorem~\ref{ca:conjugacy} therefore yields elements $x_j\in\mathcal R^2$ such that 
		$$\mathcal U_2=\prod\limits_{j=1}^{s}\exp (d_{x_j})(\mathcal U_1).$$
		
		The mode commutator identity gives $d_x(a_{(n)}b) = (d_xa)_{(n)}b+a_{(n)}d_xb$ for all $x, a,b\in\mathcal V$ and $n\in\mathbb Z$.
		The vacuum and translation axioms also give $d_x\mathbf1=0$ and $[d_x,\partial]=0,$ respectively. Thus, $d_x$ is a derivation of the vertex algebra $\mathcal V$.
		
		Since $\mathcal R$ is solvable, \cite[Corollary~8.3]{DK} leads that $\mathcal R^2$ is nilpotent. Therefore, for $x\in\mathcal R^2$, we have $d_{x}(\mathcal R)\subseteq\mathcal R^2,$ which implies the nilpotence of $d_{x}$. The iterated Leibniz identity implies
		$$\exp(d_x)(a_{(n)}b) = \bigl(\exp(d_x)a\bigr)_{(n)}
		\bigl(\exp(d_x)b\bigr), \qquad n\in\mathbb Z.$$
		Moreover, we have 
		$$\exp(d_x)\mathbf1=\mathbf1, \qquad \exp(d_x)\partial=\partial\exp(d_x), \qquad \exp(d_x)^{-1}=\exp(-d_x).$$
		Thus, each $\exp(d_{x_j})$ is an inner vertex algebra
		automorphism of $\mathcal V$, and their composition
		$\sigma$ satisfies $\sigma(\mathcal U_1)=\mathcal U_2$.
	\end{proof}
	
	\section{Structure of conformal tori.}\label{to:section}
	
	Tori of derivations play an important role in the structure theory of nilpotent Lie algebras \cite{GozeHakimjanov}, although a nilpotent Lie algebra need not admit a nonzero torus \cite{DixmierLister}. In this section we establish simultaneous diagonalization for conformal tori of finite free Lie conformal algebras and gives a criterion for the existence of a nonzero torus. 
	
	\begin{definition}\label{to:definition}
		A conformal endomorphism of a free module is called 
		\emph{diagonalizable} if its matrix is diagonal in some basis.
		A \emph{conformal torus} of a finite free {\it LCA} $\mathcal L$
		is an abelian subalgebra $\mathcal T$ of $\operatorname{CDer}(\mathcal L)$ whose elements are diagonalizable. 
		A conformal torus is said to be {\it maximal} if it is not strictly contained  in any other torus. We refer to conformal tori simply as tori.
	\end{definition}
	
	Let $s$ and $t$ be conformal endomorphisms of a free module $\mathcal V$. The conformal linearity   leads that $s_\lambda t_\mu =t_\mu s_\lambda$ if and only if
	\begin{equation}\label{to:commuting}
		s(\partial,\lambda)t(\partial+\lambda,\mu)=t(\partial,\mu)s(\partial+\mu,\lambda).
	\end{equation}
	
	\begin{lemma}\label{to:diagonal-entry} Let $s$ be a diagonalizable conformal endomorphisms of a free module $\mathcal V$ satisfying $[s_\lambda s]=0$. Then its diagonal entries in any diagonalizing basis belong to $\C[\lambda]$. 
		
		Furthermore, if 
		$\mathcal V=\mathcal U\oplus_{\C[\partial]}\mathcal W$
		is a decomposition into $s$-invariant submodules, then both 
		$s_{| \mathcal U}$ and $s_{| \mathcal W}$ are diagonalizable.
	\end{lemma}
	
	\begin{proof} Let us choose a diagonalizing basis $v_1,\ldots,v_n$, and write $s_\lambda v_i=f_i(\partial,\lambda)v_i,$ $f_i(\partial,\lambda)\in\C[\partial,\lambda].$ Applying \eqref{to:commuting} with $t=s$ gives $$f_i(\partial,\lambda)f_i(\partial+\lambda,\mu) = f_i(\partial,\mu)f_i(\partial+\mu,\lambda).$$ 
		If $f_i$ has positive $\partial$-degree $m$, with leading coefficient
		$a_m(\lambda)\ne0$, comparison of the coefficients of
		$\partial^{2m-1}$ yields $m(\lambda-\mu)a_m(\lambda)a_m(\mu)=0,$ which is a contradiction. Thus, $s_\lambda v_i=a_i(\lambda)v_i$ with $a_i\in\C[\lambda].$

		For the projection $\pi_{\mathcal U}$ onto submodule $\mathcal U$ we have $\pi_{\mathcal U}s_\lambda=s_\lambda\pi_{\mathcal U}$. Writing $\pi_{\mathcal U}(v_j)=\sum\limits_i \pi_{ij}(\partial)v_i,$ we obtain
		\begin{equation}\label{to:projection-entry}
			a_i(\lambda)\pi_{ij}(\partial+\lambda)=a_j(\lambda)\pi_{ij}(\partial).
		\end{equation}
		If $\pi_{ij}\ne0$, comparison of the leading coefficients in $\partial$
		gives $a_i=a_j$. Hence $\pi_{\mathcal U}$ preserves each summand
		$\mathcal V_a=\bigoplus\limits_{a_i=a}\C[\partial]v_i,$ where $a$ ranges over the distinct polynomials among
		$a_1,\ldots,a_n$.
		
		For $a_i=a_j=a\ne0$, equation~\eqref{to:projection-entry} reduces to $\pi_{ij}(\partial+\lambda)=\pi_{ij}(\partial),$ so $\pi_{ij}\in\C$. Thus the matrix of
		${\pi_{\mathcal U}}_{| \mathcal V_a}$ is a complex idempotent. Then a constant change of basis brings it in the form $\operatorname{diag}(I_r,0)$ and the resulting basis vectors also satisfy $s_\lambda v=a(\lambda)v$ and split into bases of
		$\pi_{\mathcal U}(\mathcal V_a)$ and
		$(\operatorname{id}-\pi_{\mathcal U})(\mathcal V_a)$.
		
		On $\mathcal V_0$, the operator $s$ vanishes, so any
		$\C[\partial]$-bases of the free modules
		$\pi_{\mathcal U}(\mathcal V_0)$ and
		$(\operatorname{id}-\pi_{\mathcal U})(\mathcal V_0)$
		diagonalize the corresponding restrictions. Finally, since $\pi_{\mathcal U}$ preserves every $\mathcal V_a$, we get 
		$$
		\mathcal U
		=
		\bigoplus_a\pi_{\mathcal U}(\mathcal V_a),
		\qquad
		\mathcal W
		=
		\bigoplus_a
		(\operatorname{id}-\pi_{\mathcal U})(\mathcal V_a).
		$$
		Combining the bases constructed above proves that both restrictions
		are diagonalizable.
	\end{proof}

	Let $\mathcal L$ be a free {\it LCA} of rank $n$ with basis $e_1,\ldots,e_n$. Suppose that $s\in\operatorname{Cend}(\mathcal L)$ satisfies $s_\lambda e_i=a_i(\lambda)e_i,$ $a_i(\lambda)\in\mathbb C[\lambda],$
	and write 
	$[e_i{}_\mu e_j]=\sum\limits_{k=1}^n f_{ij}^k(\partial,\mu)e_k.$
	The conformal Leibniz identity shows that $t$ is a conformal derivation of \(\mathcal L\) if and only if the identity
	\begin{equation}\label{to:polynomial-identity}
		a_k(\lambda)f_{ij}^k(\partial+\lambda,\mu)
		=a_i(\lambda)f_{ij}^k(\partial,\mu+\lambda) +
		a_j(\lambda)f_{ij}^k(\partial,\mu),\qquad 1\leq i,j,k\leq n
	\end{equation}
	holds.
	
	\begin{proposition}\label{to:constant-weight-reduction}
		Let $\mathcal{L}$ be a free {\it LCA} and 
		$0\neq s\in \operatorname{CDer}(\mathcal{L})$ be diagonalizable and satisfies $[s_\lambda s]=0$. Then $\mathcal L$ admits a nonzero conformal derivation whose matrix in a suitable basis is diagonal with constant entries.
	\end{proposition}
	\begin{proof} Choose a diagonalizing basis $e_1,\ldots,e_n$ of $s$. By Lemma~\ref{to:diagonal-entry} one can write $s_\lambda e_i=a_i(\lambda)e_i$, $a_i(\lambda)=\sum\limits_{r\geq0}a_{i,r}\lambda^r.$
		
		We claim that, for every $r\geq0$ and every triple $(i,j,k)$,
		the corresponding coefficients satisfy
		\begin{equation}\label{to:coefficient-weight-identity}
			a_{k,r}f_{ij}^k(\partial+\lambda,\mu)
			=
			a_{i,r}f_{ij}^k(\partial,\mu+\lambda)
			+
			a_{j,r}f_{ij}^k(\partial,\mu).
		\end{equation}
		
		Fix $(i,j,k)$. If $f_{ij}^k=0$, the claim is immediate.
		If $f_{ij}^k$ is a nonzero constant, then 
		\eqref{to:polynomial-identity} gives $a_k=a_i+a_j$, and hence
		$a_{k,r}=a_{i,r}+a_{j,r}$ for every $r\geq0$. Thus
		\eqref{to:coefficient-weight-identity} also holds in this case.
		
		Suppose now that $f_{ij}^k$ is nonconstant. Write $f_{ij}^k=\sum\limits_{p=0}^{m}(f_{ij}^k)_p,$ where $(f_{ij}^k)_p$ is homogeneous of total degree $p$ in
		$(\partial,\mu)$, with $(f_{ij}^k)_m\ne0$ and $m>0$.
		Comparing the components of total degree $m$ in
		$(\partial,\mu)$ in \eqref{to:polynomial-identity}, we obtain
		\begin{equation}\label{to:weight-sum}
			a_k=a_i+a_j.
		\end{equation}
		
		The coefficient of $\lambda$ in
		$(f_{ij}^k)_m(\partial+\lambda,\mu)$ is either zero or homogeneous of total degree $m-1$ in $(\partial,\mu)$; the same holds for $(f_{ij}^k)_m(\partial,\mu+\lambda)$. Consequently, the component of degree $m-1$ in $(\partial,\mu)$ in
		\eqref{to:polynomial-identity} is
		$$a_k\left((f_{ij}^k)_{m-1} +\lambda\bigl((f_{ij}^k)_m\bigr)_\partial
		\right) = a_i\left( (f_{ij}^k)_{m-1} +\lambda\bigl((f_{ij}^k)_m\bigr)_\mu \right) + a_j(f_{ij}^k)_{m-1}.$$
		
		By \eqref{to:weight-sum}, the terms involving
		$(f_{ij}^k)_{m-1}$ cancel. Cancelling $\lambda$ then gives
		$a_k\bigl((f_{ij}^k)_m\bigr)_\partial=a_i\bigl((f_{ij}^k)_m\bigr)_\mu.$ Substituting $a_k=a_i+a_j$, we obtain
		\begin{equation}\label{to:leading-derivative-relation}
			a_i\left(
			\bigl((f_{ij}^k)_m\bigr)_\partial
			-
			\bigl((f_{ij}^k)_m\bigr)_\mu
			\right)
			+
			a_j\bigl((f_{ij}^k)_m\bigr)_\partial
			=0.
		\end{equation}
		
		Consider the $\mathbb C$-linear map
		$F:\mathbb C^2\longrightarrow\mathbb C[\partial,\mu]$ defined by
		\[
		F(x,y)
		=
		x\left(
		\bigl((f_{ij}^k)_m\bigr)_\partial
		-
		\bigl((f_{ij}^k)_m\bigr)_\mu
		\right)
		+
		y\bigl((f_{ij}^k)_m\bigr)_\partial.
		\]
		Since $(f_{ij}^k)_m$ has positive degree, its two partial
		derivatives cannot both vanish. Thus $F\ne0$, and therefore
		$\dim_{\mathbb C}\ker F\leq1$.
		
		Comparing coefficients of $\lambda$ in
		\eqref{to:leading-derivative-relation} gives
		\[
		a_{i,r}\left(
		\bigl((f_{ij}^k)_m\bigr)_\partial
		-
		\bigl((f_{ij}^k)_m\bigr)_\mu
		\right)
		+
		a_{j,r}\bigl((f_{ij}^k)_m\bigr)_\partial
		=0
		\]
		for every $r\geq0$. Hence
		$(a_{i,r},a_{j,r})\in\ker F$ for every $r\geq0$.
		
		If all these pairs vanish, then $a_i=a_j=0$, and
		\eqref{to:weight-sum} implies $a_k=0$. In this case, every
		coefficient triple is zero, so
		\eqref{to:coefficient-weight-identity} holds trivially. Otherwise, choose $r_0$ such that
		$(a_{i,r_0},a_{j,r_0})\ne(0,0)$. Since
		$\dim_{\mathbb C}\ker F\leq1$, for each $r\geq0$ there is a
		scalar $c_r\in\mathbb C$ such that $(a_{i,r},a_{j,r})
		=c_r(a_{i,r_0},a_{j,r_0}).$ By \eqref{to:weight-sum}, we get
		$$a_{k,r}=a_{i,r}+a_{j,r}=c_r(a_{i,r_0}+a_{j,r_0})=c_r a_{k,r_0}.$$
		Thus, setting $q(\lambda)=\sum\limits_{r\geq0}c_r\lambda^r$, we obtain $\bigl(a_i(\lambda),a_j(\lambda),a_k(\lambda)\bigr)
		=q(\lambda)\bigl(a_{i,r_0},a_{j,r_0},a_{k,r_0}\bigr).$
		Substituting this into \eqref{to:polynomial-identity} yields
		\[
		q(\lambda)\Bigl(
		a_{k,r_0}f_{ij}^k(\partial+\lambda,\mu)
		-
		a_{i,r_0}f_{ij}^k(\partial,\mu+\lambda)
		-
		a_{j,r_0}f_{ij}^k(\partial,\mu)
		\Bigr)=0.
		\]
		Since $q\ne0$ (because of $c_{r_0}=1$),
		we may cancel $q(\lambda)$. Thus the coefficient triple
		corresponding to $r_0$ satisfies
		\eqref{to:coefficient-weight-identity}. Every other coefficient
		triple is a scalar multiple of this one, so
		\eqref{to:coefficient-weight-identity} holds for every $r\geq0$.
		This proves the claim.
		
		For each $r\geq0$, define a conformal endomorphism $d^{(r)}$ by
		$d^{(r)}_\lambda e_i=a_{i,r}e_i,$ $1\leq i\leq n.$ Then By \eqref{to:coefficient-weight-identity}, we have
		$d^{(r)}\in\operatorname{CDer}(\mathcal L)$ for every $r\geq0$.
		Finally, since $s\ne0$, at least one coefficient vector
		$(a_{1,r},\ldots,a_{n,r})$ is nonzero. Choosing such an $r$
		and setting $d=d^{(r)}$ gives the required nonzero conformal
		derivation with diagonal constant matrix.
	\end{proof}
	
	Similar as in \cite{DM2012} for a torus $\mathcal T$ of a free {\it LCA} $\mathcal L$ we set 
	$$\mathcal L_\alpha=\operatorname{Span}_{\C[\partial]}\{v\in\mathcal L\mid t_\lambda v=\alpha_t(\lambda)v\text{ for every }t\in\mathcal T\}.$$
	If $\mathcal L_\alpha\neq 0$, then $\alpha$ is a {\it weight} for $\mathcal{L}$ on $\mathcal{T}.$ We denote the set of weights by $W$.
	By Theorem~\ref{to:simultaneous}, every weight in $W$ occurs
	among the weights of a common diagonalizing basis of $\mathcal L$.
	Hence $|W|\leq\operatorname{rank}(\mathcal L)$.
	
	\begin{corollary}\label{to:weight-decomposition}
		One has $\mathcal L=\bigoplus\limits_{\alpha\in W}\mathcal L_\alpha$ and $[{\mathcal L_\alpha}\cdot{\mathcal L}_\beta]\subseteq\mathcal L_{\alpha+\beta}.$
	\end{corollary}
	\begin{proof} By Theorem~\ref{to:simultaneous}, choose a common diagonalizing basis $e_1,\ldots,e_n$ with weights $\alpha_1,\ldots,\alpha_n$. Grouping the basis vectors by their weights gives $\mathcal L=\bigoplus\limits_{\alpha\in W}\mathcal L_\alpha$.
		
		If structure constant $f_{ij}^k\ne0$, the argument yielding \eqref{to:weight-sum} gives $(\alpha_k)_t=(\alpha_i)_t+(\alpha_j)_t,$ $t\in\mathcal T,$ so $\alpha_k=\alpha_i+\alpha_j$.
		The bracket inclusion now follows by conformal sesquilinearity.
	\end{proof}
	
	We now establish simultaneous diagonalization for abelian
	submodules of $\operatorname{Cend}(\mathcal V)$ whose elements
	are diagonalizable, together with a bound on their rank.

	\begin{theorem}\label{to:simultaneous} Let $\mathcal V$ be a free module of rank $n$, and let $\mathcal S$ be a submodule of $\operatorname{Cend}(\mathcal V)$ whose elements are pairwise commute and diagonalizable. Then there exists a basis of $\mathcal V$ that simultaneously diagonalizes every element of $\mathcal S$ and $\operatorname{rank}(\mathcal S)\leq n$.
	\end{theorem}
	
	\begin{proof}
		We prove simultaneous diagonalization by induction on $n$.
		The assertion is immediate if $n\leq1$ or $\mathcal S=0$.
		
		For $t\in\mathcal S$, choose a diagonalizing basis
		$e_1,\ldots,e_n$. By Lemma~\ref{to:diagonal-entry}, we have $t_\lambda e_i=a_i(\lambda)e_i$ for $a_i\in\C[\lambda].$ The leading-coefficient argument in the proof of
		Lemma~\ref{to:diagonal-entry} shows that each weight summand
		$\mathcal V[a]=\bigoplus\limits_{a_i=a}\C[\partial]e_i$
		and its complementary sum are $\mathcal S$-invariant.
		The same lemma then implies that every element of $\mathcal S$
		restricts to a diagonalizable conformal endomorphism of
		$\mathcal V[a]$.
		If some $t\in\mathcal S$ has at least two distinct weights,
		each summand has rank less than $n$. Applying the induction
		hypothesis to the images of $\mathcal S$ under restriction
		and combining the resulting bases proves simultaneous
		diagonalization in this case.
		
		Therefore we can assume that every nonzero element of
		$\mathcal S$ has a scalar diagonal form $c(\lambda)I_n$,
		where $c\in\C[\lambda]\setminus\{0\}$.
		
		Let fix $0\ne t\in\mathcal S$ and a basis in which
		$t_\lambda e_i=a(\lambda)e_i$, with $a\ne0$.
		Then for every $s\in\mathcal S$ the commutativity with $t$ implies that the matrix of $s$ is independent of $\partial$; denote it by $S_s(\lambda)$.
		Moreover, 
		$$S_s(\lambda)S_r(\mu)=S_r(\mu)S_s(\lambda)
		\qquad r, s\in\mathcal S.$$
		Thus all coefficient matrices commute and have a common
		eigenvector $0\ne v\in\operatorname{span}_{\C}\{e_1,\ldots,e_n\}.$ 
		For each $s\in\mathcal S$, write
		$s_\lambda v=b_s(\lambda)v$, with $b_s\in\C[\lambda]$. 
		Conformal sesquilinearity gives $u_\lambda v=0$ for $u=a(-\partial)s-b_s(-\partial)t\in\mathcal S.$ Since every nonzero element of $\mathcal S$ has zero kernel, we obtain $u=0$. Consequently, $S_s(\lambda)=b_s(\lambda)I_n$.
		Hence the chosen basis diagonalizes every element of
		$\mathcal S$.
		
		Finally, in a common diagonalizing basis $e_1,\ldots,e_n$,
		write $s_\lambda e_i=c_i^s(\lambda)e_i,$ $c_i^s\in\C[\lambda],$
		where independence of $\partial$ follows from
		Lemma~\ref{to:diagonal-entry}. The map $\varphi:\mathcal S\rightarrow\C[\lambda]^n,$ $s\mapsto
		\bigl(c_1^s(\lambda),\ldots,c_n^s(\lambda)\bigr),$
		is injective. 
		
		It is also $\C[\partial]$-linear when $\partial$ acts on
		$\C[\lambda]$ by multiplication by $-\lambda$.
		Since $\C[\lambda]^n$ is a free $\C[\partial]$-module of rank $n$
		and $\C[\partial]$ is a PID, its submodule $\varphi(\mathcal S)$,
		and hence $\mathcal S$, is free. Therefore,
		$\operatorname{rank}(\mathcal S)\leq n$. 
	\end{proof}
	
	For a maximal torus, the constant-weight reduction yields
	a basis with constant diagonal entries, as the following
	corollary shows.
	
	\begin{corollary}\label{to:maximal-constant-basis}
		Let $\mathcal T$ be a maximal torus of a free
		{\it LCA} $\mathcal L$ of rank $n$. Then $\mathcal T$ is free of rank at most $n$ and admits a basis whose elements have constant diagonal matrices in a common basis of $\mathcal L$.
	\end{corollary}
	\begin{proof} By Theorem~\ref{to:simultaneous}, we have
		$\operatorname{rank}(\mathcal T)\leq n$, and we may fix a common
		diagonalizing basis $e_1,\ldots,e_n$ of $\mathcal L$.
		Due to Lemma \ref{to:diagonal-entry} we get $t_\lambda e_i=\sum\limits_r a_{i,r}\lambda^r e_i, \ t\in \mathcal T$. The proof of Proposition~\ref{to:constant-weight-reduction} shows that $(d_r)_\lambda e_i=a_{i,r}e_i, \ a_{i,r}\in \C$ defines a conformal derivation. Since $\mathcal T+\C[\partial]d_r$ is again a torus, the maximality of $\mathcal T$ gives $d_r\in\mathcal T$. Since $t=\sum\limits_r(-\partial)^r d_r$, the elements of $\mathcal T$ having constant diagonal matrices generate $\mathcal T$ over $\C[\partial]$. Choose a complex basis of the finite-dimensional space formed by these elements. Comparing coefficients of $\lambda$ shows that this basis is also $\C[\partial]$-linearly independent, and hence forms a basis of $\mathcal T$. 
	\end{proof}
	
	The constant-weight reduction gives a criterion for the existence
	of a nonzero torus in terms of the structure coefficients of
	$\mathcal L$. In a suitable basis, the same identities determine
	the rank of a maximal torus.
	
	\begin{theorem}\label{to:constant-weight-criterion}
		A free {\it LCA} $\mathcal L$ admits a nonzero torus if and only if there exist a basis $e_1,\ldots,e_n$ and $0\ne(w_1,\ldots,w_n)\in\C^n$ such that whenever $f_{i,j}^k\ne0$,
		\begin{equation}\label{eq5.71}
			w_k=w_i+w_j,\qquad w_k(f_{i,j}^k)_\partial=w_i(f_{i,j}^k)_\mu,\qquad w_k(f_{i,j}^k)_{\partial\partial}=w_i(f_{i,j}^k)_{\mu\mu}.
		\end{equation}
		Moreover, for any maximal torus $\mathcal T$, its rank equals the dimension of the solution space of these identities in $\C^n$, relative to a common basis supplied by Corollary~\ref{to:maximal-constant-basis}.
	\end{theorem}
	
	\begin{proof}
		Suppose that $\mathcal L$ admits a nonzero torus. By Proposition~\ref{to:constant-weight-reduction}, choose a basis $e_1,\ldots,e_n$ and a nonzero conformal derivation $d_\lambda e_i=w_i e_i$, with $w_i\in\C$. The weight rule gives $w_k=w_i+w_j$ whenever $f_{i,j}^k\ne0$. Taking $a_i(\lambda)=w_i$ in \eqref{to:polynomial-identity}, differentiating once and twice with respect to $\lambda$, and setting $\lambda=0$ gives the remaining identities.
		
		Conversely, suppose that these identities hold. Define $d\in\operatorname{Cend}(\mathcal L)$ by $d_\lambda e_i=w_i e_i$. It suffices to verify \eqref{to:polynomial-identity} with $a_i(\lambda)=w_i$. Let us fix $0\neq f_{i,j}^k$. 
		
		If $w_i=w_j=0$, then $w_k=0$, and the identity is immediate.
		If exactly one of $w_i,w_j,w_k$ vanishes, the relation
		$w_k=w_i+w_j$ and the identity $w_kf_\partial=w_if_\mu$ give
		\[
		f\in
		\begin{cases}
			\C[\mu],& w_i=0,\\
			\C[\partial+\mu],& w_j=0,\\
			\C[\partial],& w_k=0.
		\end{cases}
		\]
		In each case, substitution verifies the required identity.
		
		It remains to consider the case $w_iw_jw_k\ne0$.
		Differentiating $w_kf_\partial=w_if_\mu$ with respect to
		$\partial$ and $\mu$ yields $w_kf_{\partial\partial}=w_if_{\partial\mu}$ and $w_kf_{\partial\mu}=w_if_{\mu\mu}.$ Hence, we get $w_k^2f_{\partial\partial}=w_i^2f_{\mu\mu}.$ Combining this with
		$w_kf_{\partial\partial}=w_if_{\mu\mu}$, we obtain
		$$w_k(w_k-w_i)f_{\partial\partial}=w_kw_jf_{\partial\partial}
		=0.$$ Thus, $f_{\partial\partial}=0$, and the above identities imply 
		$f_{\partial\mu}=f_{\mu\mu}=0$. So, $f$ is linear with respect to $\partial$ and $\mu$. Along with $w_kf_\partial=w_if_\mu$ it gives
		$$f(\partial,\mu)=A(w_i\partial+w_k\mu)+B, \qquad A,B\in\C.$$
		This expression also satisfies \eqref{to:polynomial-identity}.
		Consequently, $d$ is a nonzero conformal derivation. Since its
		matrix is constant and diagonal, $[d_\lambda d]=0$, and
		$\C[\partial]d$ is a nonzero torus.
		
		Finally, let $\mathcal T$ be maximal. By Corollary~\ref{to:maximal-constant-basis}, choose bases $t_1,\ldots,t_s$ of $\mathcal T$ and $e_1,\ldots,e_n$
		of $\mathcal L$ such that $(t_r)_\lambda e_i=w_{i,r}e_i$ with  $w_{i,r}\in\C$ and $1 \leq r \leq s.$ The above argument shows that each vector $W_r:=(w_{1,r},\ldots,w_{n,r})$ satisfies \eqref{eq5.71}. These vectors are linearly independent over $\C$ because $t_1,\ldots,t_s$ form a module basis.
		
		Now let $(w_1,\ldots,w_n)$ be any solution of \eqref{eq5.71}. Define a conformal endomorphism $d$ by
		$d_\lambda e_i=w_i e_i, \ w_i\in \C.$ The sufficiency argument above shows that $d$ is a conformal derivation. Moreover, the equalities $[d_\lambda t_r]=0$ and $[d_\lambda d]=0$ imply that $\mathcal T+\C[\partial]d$ is abelian. Moreover, every element of the submodule $\mathcal T+\C[\partial]d$ is diagonal in the basis $e_1,\ldots,e_n$. Consequently, $\mathcal T+\C[\partial]d$ is a torus containing $\mathcal T$. By maximality of $\mathcal T$, it follows that $d\in\mathcal T$. Hence,  $d=\sum\limits_{r=1}^sp_r(\partial)t_r, \ p_r\in\C[\partial]$. 
		The module action gives
		$$w_i e_i=d_\lambda e_i=\sum\limits_{r=1}^s p_r(-\lambda)(t_r)_\lambda e_i=\left(\sum\limits_{r=1}^s p_r(-\lambda)w_{i,r}\right)e_i.$$
		Comparing the coefficients of $e_i$ and setting $\lambda=0$ yields $(w_1,\ldots,w_n)=\sum\limits_{r=1}^s p_r(0)W_r.$ 
		Since each $p_r(0)$ is a complex number, every solution belongs to the complex span of $W_r, \ 1\le r \le s$. Together with their linear independence, this proves that they form a basis of the solution space. Its dimension is therefore $\operatorname{rank}(\mathcal T)=s$.
	\end{proof}
	
	The following example shows that the conjugacy of maximal tori of finite-dimensional complex nilpotent Lie algebras \cite{Mostow1956} does not extend to conformal tori of finite nilpotent {\it LCA}s.
	
	\begin{example}\label{to:different-ranks}
		The nilpotent {\it LCA} $\operatorname{Cur}(\mathfrak n_4)$,
		where $\mathfrak n_4$ is the Lie algebra of strictly upper triangular $4\times4$ complex matrices, admits maximal tori that are not conjugated.
		
		Indeed, let $g_1=e_{12}, \ g_2=e_{13}, \  g_3=e_{14}, \
		g_4=e_{23}, \ g_5=e_{24}, \ g_6=e_{34},$ where $e_{ij}$ are the standard matrix units. These elements form a
		$\mathbb C[\partial]$-basis of $\operatorname{Cur}(\mathfrak n_4)$,
		and the nonzero $\lambda$-brackets, up to skew-symmetry, are
		$$[g_1{}_\lambda g_4]=g_2,\qquad
		[g_1{}_\lambda g_5]=g_3,\qquad
		[g_2{}_\lambda g_6]=g_3,\qquad
		[g_4{}_\lambda g_6]=g_5.$$
		Relative to this basis, consider the diagonal operators
		$$h_1=\operatorname{diag}(1,1,1,0,0,0), \quad h_2=\operatorname{diag}(-1,0,0,1,1,0), \quad 
		h_3=\operatorname{diag}(0,-1,0,-1,0,1).$$
		A direct verification shows that $\mathcal T=\bigoplus\limits_{i=1}^{3}\mathbb C[\partial]h_i$ is a maximal torus of rank $3$.
		
		If we consider the $\mathbb C[\partial]$-basis
		$$f_1=g_1+\partial g_2, \quad f_2=g_2, \quad f_3=g_3, \quad f_4=g_4, \quad f_5=g_5, \quad f_6=g_6+2\partial g_5,$$ 
		then with respect to this basis, the diagonal operator
		$t=\operatorname{diag}(-2,-2,-1,0,1,1)$ generates a maximal torus
		$\mathcal T'=\mathbb C[\partial]t$ of rank $1$, as can again be checked directly.
		
		Since conjugation preserves $\mathbb C[\partial]$-rank, the maximal tori $\mathcal T'$ and $\mathcal T$ are not conjugate.
	\end{example}

	\section{Null-filiform Lie conformal algebras: classification and maximal tori}\label{nf:section}
	
	We begin this section by introducing the relevant definitions and a family of null-filiform Lie conformal algebras, and then investigate their conformal tori and solvable extensions.
	
	Let $\mathcal N$ be a nilpotent {\it LCA} with nilindex $s+1$. Define $\mathcal{N}_i = \mathcal{N}^i / \mathcal{N}^{i+1},$ $1 \le i \le s,$ and set
	$\operatorname{gr}(\mathcal{N})=\bigoplus\limits_{i=1}^s\mathcal{N}_i.$ Then $[\mathcal{N}_i \cdot \mathcal{N}_j] \subseteq \mathcal{N}_{i+j}$, so that $\operatorname{gr}(N)$ becomes a $\mathbb{Z}$-graded algebra.  We say that $\mathcal{N}$ is {\it naturally graded algebra} if $\operatorname{gr}(\mathcal{N})\cong \mathcal{N}$. 
	
	We call $\mathcal N$ \emph{totally saturated} if every term
	$\mathcal N^k$ of its lower central series is saturated
	in $\mathcal N$.
	
	\begin{proposition}\label{prop6.1}
		Let $\mathcal N$ be a nilpotent Lie conformal algebra that is
		finitely generated as a $\mathbb C[\partial]$-module. The following
		conditions are equivalent:
		\begin{enumerate}
			\item $\mathcal N$ is totally saturated;
			\item $\mathcal N^{k+1}$ is saturated in $\mathcal N^k$ for every
			$k\geq1$;
			\item $\mathcal N^k/\mathcal N^{k+1}$ is free over
			$\mathbb C[\partial]$ for every $k\geq1$.
		\end{enumerate}
		Under these equivalent conditions, $\mathcal N$ is free over
		$\mathbb C[\partial]$.
	\end{proposition}
	
	\begin{proof} If $\mathcal N$ is totally saturated,
		then
		$$\operatorname{Sat}_{\mathcal N^k}(\mathcal N^{k+1})
		=\mathcal N^k\cap\operatorname{Sat}_{\mathcal N}(\mathcal N^{k+1})=\mathcal N^{k+1}, \quad \mbox{for every} \quad k\geq1.$$
		Thus~$(1)$ implies~$(2)$.
		
		Conversely, assume~$(2)$. We prove by induction on $k$ that
		$\mathcal N^k$ is saturated in $\mathcal N$. The case $k=1$
		is immediate. Suppose the assertion holds for $k$, and let
		$x\in\mathcal N$ satisfy
		$p(\partial)x\in\mathcal N^{k+1}$ for some nonzero $p\in \mathbb C[\partial]$.
		The inclusion $\mathcal N^{k+1}\subseteq\mathcal N^k$ and the
		induction hypothesis imply that $x\in\mathcal N^k$.
		Condition~$(2)$ then gives $x\in\mathcal N^{k+1}$, completing
		the induction and proving~$(1)$.
		
		To prove that~$(2)$ implies~$(3)$, fix $k\geq1$ and write
		$Q_k=\mathcal N^k/\mathcal N^{k+1}$. Since $\mathbb C[\partial]$ is Noetherian
		and $\mathcal N$ is finitely generated, $Q_k$ is finitely
		generated. If $x\in\mathcal N^k$ and $0\neq p\in \mathbb C[\partial]$ satisfy $p(\partial)(x+\mathcal N^{k+1})=0,$
		then $p(\partial)x\in\mathcal N^{k+1}$. By~$(2)$,
		$x\in\mathcal N^{k+1}$, so $Q_k$ is torsion-free.
		The structure theorem for finitely generated modules over a principal ideal domain (see \cite[\S 12.1]{DummitFoote}) therefore implies that $Q_k$ is free.
		
		Conversely, assume~$(3)$. Each $Q_k$ is then torsion-free.
		Consequently, for $x\in\mathcal N^k$ and $0\neq p\in \mathbb C[\partial]$,
		the relation $p(\partial)x\in\mathcal N^{k+1}$ forces
		$x+\mathcal N^{k+1}=0$. Hence $\mathcal N^{k+1}$ is saturated
		in $\mathcal N^k$, proving~$(2)$.
		
		Finally, nilpotence gives $\mathcal N^{c+1}=0$ for some $c$.
		By~$(1)$, the zero submodule is saturated in $\mathcal N$,
		so $\mathcal N$ is torsion-free. Since $\mathcal N$ is
		finitely generated over the principal ideal domain $\mathbb C[\partial]$,
		it is free.
	\end{proof}
	
	\begin{definition}
		A free $\mathcal N$ of rank $n$ is called \emph{null-filiform} if $\operatorname{rank}(\mathcal N^k)=n+1-k$ for $1\leq k\leq n+1$. A basis $\{e_1, \ldots, e_n\}$ of a null-filiform algebra $\mathcal N$ is called \emph{adapted} if $\mathcal N^k=\bigoplus\limits_{r=k}^n\C[\partial]e_r$ for $1\leq k\leq n.$
	\end{definition}
	
	\begin{remark}\label{rem6.3} Note that a null-filiform LCA $\mathcal N$ admits an adapted basis if and only if it is totally saturated. Indeed, an adapted basis represents each $\mathcal N^k$ as a direct summand of $\mathcal N$, and hence as a saturated submodule. Conversely, if $\mathcal N$ is totally saturated, then Proposition~\ref{prop6.1} leads each $\mathcal N^i/\mathcal N^{i+1}$ is free of rank one. Choose $e_i\in\mathcal N^i$ whose class generates this quotient. Then $\mathcal N^i=\C[\partial]e_i\oplus\mathcal N^{i+1}.$ Since $\mathcal N^{n+1}=0$, successive substitution gives $\mathcal N^k=\bigoplus\limits_{i=k}^n\C[\partial]e_i$ for every $k$, so $\{e_1,\ldots,e_n\}$ is an adapted basis.
	\end{remark}
	
	Throughout this section, $\mathcal N$ denotes a null-filiform {\it LCA} of rank $n\geq 2$.
	
	\begin{lemma}\label{nf:saturation-brackets}
		Let $\mathcal U$ and $\mathcal V$ be submodules of
		$\mathcal L$. Then
		\begin{enumerate}
			\item[(i)]
			$[\operatorname{Sat}_{\mathcal L}(\mathcal U)\cdot\mathcal V]
			=[\mathcal U\cdot\mathcal V]$;
			\item[(ii)]
			$[\operatorname{Sat}_{\mathcal L}(\mathcal U)\cdot
			\operatorname{Sat}_{\mathcal L}(\mathcal V)]
			=[\mathcal U\cdot\mathcal V]$.
		\end{enumerate}
	\end{lemma}
	\begin{proof}
		To prove~(i), let $u\in\operatorname{Sat}_{\mathcal L}(\mathcal U)$ and $v\in\mathcal V$. Choose a nonzero polynomial
		$p\in\C[\partial]$ such that $p(\partial)u\in\mathcal U$. Since $[p(\partial)u_\lambda v]=p(-\lambda)[u_\lambda v]$ and $p(-\lambda)$ has nonzero scalar leading coefficient, the coefficients of $[u_\lambda v]$ and those of $p(-\lambda)[u_\lambda v]$ generate the same submodule; see also \cite[Lemma~3.1]{DK}. 
		Consequently, every coefficient of $[u_\lambda v]$ belongs to
		$[\mathcal U\cdot\mathcal V]$, proving $[\operatorname{Sat}_{\mathcal L}(\mathcal U)\cdot\mathcal V]
		\subseteq[\mathcal U\cdot\mathcal V].$
		The reverse inclusion follows from
		$\mathcal U\subseteq\operatorname{Sat}_{\mathcal L}(\mathcal U)$.
		
		For (ii), applying (i) and conformal skew-symmetry gives
		$$[\operatorname{Sat}_{\mathcal L}(\mathcal U)\cdot
		\operatorname{Sat}_{\mathcal L}(\mathcal V)]
		=[\mathcal U\cdot\operatorname{Sat}_{\mathcal L}(\mathcal V)]
		=[\mathcal U\cdot\mathcal V].
		$$
	\end{proof}
	
	We now show that a null-filiform {\it LCA} with a nonzero conformal torus is totally saturated and that each nonzero torus element admits an adapted basis in which it acts diagonally.
	
	\begin{proposition}\label{nf:torus-saturation}
		Let $\mathcal N$ admits a nonzero conformal torus. Then $\mathcal N$ is totally saturated. Moreover, for every nonzero element $t$ of a torus of $\mathcal N$, there exist an adapted basis $\{v_1,\ldots,v_n\}$ and a nonzero polynomial $a\in\C[\lambda]$ such that $t_\lambda v_i=i\,a(\lambda)v_i$, $1\leq i\leq n.$
	\end{proposition}
	\begin{proof}
		Fix a nonzero element $t$ of a conformal torus of $\mathcal N$, and choose a basis in which $t$ is diagonal. By Lemma~\ref{to:diagonal-entry}, its diagonal entries belong to $\C[\lambda]$. The argument in the proof of Proposition~\ref{to:constant-weight-reduction} yields a nonzero conformal derivation $d$ whose matrix in the same basis is diagonal with constant entries. In particular, the derivation
		$d_0:={d_\lambda}_{| \lambda=0}$ is diagonalizable.
		
		Set
		$$F^k=\operatorname{Sat}_{\mathcal N}(\mathcal N^k),
		\quad 1\leq k\leq n+1, \quad \mbox{and} \quad 
		Q_k=F^k/F^{k+1}, \quad 1\leq k\leq n.$$
		Saturation preserves rank, and $F^{k+1}$ is saturated in $F^k$. The null-filiform rank condition therefore implies that each $Q_k$ is free of rank one. Moreover, $F^{n+1}=0$. By Lemma~\ref{nf:saturation-brackets}, we have 
		$$[F^i\cdot F^j]=[\mathcal N^i\cdot\mathcal N^j]
		\subseteq \mathcal N^{i+j}\subseteq F^{i+j}.$$
		Thus the bracket induces a conformal bilinear map from $Q_1\times Q_k$ to $Q_{k+1}$ for $1\leq k<n$.
		
		The derivation $d_0$ preserves every $\mathcal N^k$, and hence every $F^k$. Its spectral projections consequently preserve $F^k$ and $F^{k+1}$. Since $Q_k$ has rank one, the induced action of $d_0$ on $Q_k$ is multiplication by a scalar $w_k\in\C$.
		
		For $1\leq k<n$, the induced bracket $[Q_1\cdot Q_k]$ is nonzero. Indeed, its vanishing would imply
		$$\mathcal N^{k+1}=[\mathcal N\cdot\mathcal N^k]\subseteq [F^1\cdot F^k]\subseteq F^{k+2},$$
		contrary to $\operatorname{rank}(\mathcal N^{k+1}) > \operatorname{rank}(F^{k+2}).$ The derivation identity therefore gives $w_{k+1}=w_1+w_k,$ for $1\leq k<n,$
		and hence $w_k=kw_1$ for every $k$. If $w_1=0$, then
		$d_0(F^k)\subseteq F^{k+1}$ for all $k$, so $d_0^n=0$. Since $d_0$ is diagonalizable, this forces $d_0=0$. As $d$ has constant diagonal entries, it follows that $d=0$, a contradiction. Thus $w_1\neq0$, and the weights $w_1,\ldots,w_n$ are nonzero and pairwise distinct. We may therefore reorder the chosen basis as $\{v_1,\ldots,v_n\}$ so that
		$$d_\lambda v_j=w_jv_j, \qquad t_\lambda v_j=a_j(\lambda)v_j,
		\qquad a_j\in\C[\lambda].$$
		The weights occurring in $F^k$ are precisely $w_k,\ldots,w_n$. Since $F^k$ is saturated and $d_0$ acts on $F^r/F^{r+1}$ as multiplication by $w_r$ for $k\leq r\leq n$, we have 
		$F^k=\bigoplus\limits_{j=k}^n\C[\partial]v_j.$

		Since $\mathcal N^k\subseteq F^k$ is $d_0$-invariant and the eigenvalues of $d_0$ are pairwise distinct, we have
		$$\mathcal N^k=\bigoplus_{j=k}^{n}\bigl(\mathcal N^k\cap\C[\partial]v_j\bigr)=\bigoplus_{j=k}^{n}
		p_{k,j}(\partial)\C[\partial]v_j$$
		for suitable nonzero polynomials $p_{k,j}$ (their nonvanishing follows from $\operatorname{rank}(\mathcal N^k)=\operatorname{rank}(F^k)$).
		Taking into account that $d$ preserves $\mathcal N^k$, we obtain 
		$$d_\lambda\bigl(p_{k,j}(\partial)v_j\bigr) = w_jp_{k,j}(\partial+\lambda)v_j \in\mathcal N^k[\lambda].$$
		The coefficient of the highest power of $\lambda$ is a nonzero scalar multiple of $v_j$, since $w_j\neq0$. Comparing coefficients yields $v_j\in\mathcal N^k$ (cf.~\cite[Lemma~3.1]{DK}). Thus $\mathcal N^k=F^k=\bigoplus\limits_{j=k}^{n}\C[\partial]v_j.$
		It follows that $\mathcal N$ is totally saturated and that $\{v_1,\ldots,v_n\}$ is an adapted basis.
		
		Finally, the nonvanishing of $[Q_1\cdot Q_k]$ and \eqref{to:weight-sum} imply $a_{k+1}=a_1+a_k$ for $1\leq k<n$.
		Therefore, we have $a_k=k\,a_1$ for every $k$. Since $t\neq0$, we have $a_1\neq0$. Taking $a=a_1$ completes the proof.
	\end{proof}
	
	The following result gives a criterion for a saturated null-filiform Lie conformal algebra to admit a nonzero torus and provides a structural description of such algebras in an adapted basis.
	
	\begin{proposition}\label{to:null-torus-existence}
		The algebra $\mathcal N$ admits a nonzero conformal torus if and
		only if it admits an adapted basis $\{v_1, \ldots, v_n\}$ 
		such that $d_\lambda v_i=i v_i$ defines a conformal derivation of $\mathcal N$. In any such basis, the multiplication table of $\mathcal N$ is given by
		\begin{equation}\label{eq6.1.01}
			[v_i{}_\mu v_j]=\bigl(c_{ij}(i\partial+(i+j)\mu)+b_{ij}\bigr)v_{i+j} \quad 2 \leq i+j\leq n,
		\end{equation} where $c_{ij}, b_{ij}\in\C$.
	\end{proposition}
	\begin{proof} Suppose that $\mathcal N$ admits a nonzero torus. By Proposition \ref{nf:torus-saturation}, a nonzero element $t$ of this torus acts in an adapted basis by $t_\lambda v_i=i\,a(\lambda)v_i,$ $0\ne a\in\C[\lambda].$ Define $d\in\operatorname{Cend}(\mathcal N)$ by $d_\lambda v_i=i v_i$. Since $t_\lambda=a(\lambda)d_\lambda$, cancelling $a(\lambda)$ in \eqref{to:polynomial-identity} for $t$ shows that $d$ is a conformal derivation. Conversely, such a derivation generates the nonzero torus $\C[\partial]d$.
		
		Now fix any adapted basis with this property and write $[v_i{}_\mu v_j]=\sum\limits_{k=1}^{n} f_{i,j}^k(\partial,\mu)v_k.$
		By Corollary~\ref{to:weight-decomposition}, $f_{i,j}^k=0$ unless $k=i+j$. 
		
		For $k=i+j\leq n$, equation~\eqref{eq5.71} gives $k(f_{i,j}^k)_\partial=i(f_{i,j}^k)_\mu,$ and $k(f_{i,j}^k)_{\partial\partial}=i(f_{i,j}^k)_{\mu\mu}.$ 
		Differentiating the first identity with respect to
		$\partial$ and $\mu$, and using the second, yields
		$$k^2(f_{i,j}^k)_{\partial\partial}=i^2(f_{i,j}^k)_{\mu\mu}=ik(f_{i,j}^k)_{\partial\partial}.$$
		Since $k(k-i)=kj\ne0$, we obtain $(f_{i,j}^k)_{\partial\partial}=0$. The second identity then gives $(f_{i,j}^k)_{\mu\mu}=0$, while the $\partial$-derivative of the first gives $(f_{i,j}^k)_{\partial\mu}=0$. Thus $f_{i,j}^k$ has total degree at most one, and the first identity implies $f_{i,j}^k(\partial,\mu)=c_{ij}(i\partial+k\mu)+b_{ij}$
		for some $c_{ij},b_{ij}\in\C$. Since $k=i+j$, this proves \eqref{eq6.1.01}.
	\end{proof}
	Note that the algebras described by products \eqref{eq6.1.01} are quadratic Lie conformal algebras in the sense of \cite{Xu2000}.
	
	\begin{example}
		Consider the standard generators $\{L_i\}_{i\geq0}$ of the Block-type LCA $\mathfrak B(q)$, $q\in\mathbb N$, set
		$$\mathfrak B_n(q):=\left(\bigoplus\limits_{i\geq1}\C[\partial]L_i\right)
		\Big/\left(\bigoplus\limits_{i\geq n+1}\C[\partial]L_i\right), \quad n\geq3$$
		and denote by $e_i$ the image of $L_i$ in $\mathfrak B_n(q)$ for $1\leq i\leq n$. The lower central series of $\mathfrak B_n(q)$ is given by $\mathfrak B_n(q)^k=\bigoplus\limits_{i=k}^{n}\C[\partial]e_i$ for $1\leq k\leq n+1.$ Then $\mathfrak B_n(q)$ is null-filiform and totally saturated, with adapted basis $\{e_1,\ldots,e_n\}$.
		
		Suppose that $\mathfrak B_n(q)$ admits a nonzero conformal torus. By Proposition~\ref{to:null-torus-existence}, there is an adapted basis $\{v_1,\ldots,v_n\}$ such that 
		$[v_1{}_\lambda v_2]=\bigl(\alpha(\partial+3\lambda)+\beta\bigr)v_3$ for some $\alpha,\beta\in\C$. Since both bases are adapted, we may write
		$v_i\equiv a_i e_i\pmod{\mathfrak B_n(q)^{i+1}}$ with $a_i\in\C^\times.$  The defining brackets of $\mathfrak B(q)$ then give 
		$$[v_1{}_\lambda v_2]\equiv\frac{a_1a_2}{a_3}\big((q+1)\partial+(2q+3)\lambda\big)v_3\pmod{\mathfrak B_n(q)^4[\lambda]}.$$
		Comparing the coefficients of $\partial$ and $\lambda$ in the free rank-one quotient $\mathfrak B_n(q)^3/\mathfrak B_n(q)^4$, we obtain
		$$\alpha=\frac{a_1a_2}{a_3}(q+1),\qquad 3\alpha=\frac{a_1a_2}{a_3}(2q+3).$$
		These identities imply $q=0$, a contradiction. Therefore $\mathcal N$ admits no nonzero conformal torus. In particular, a totally saturated null-filiform LCA need not admit a nonzero conformal torus.
	\end{example}
	
	We introduce the following three families of totally saturated
	null-filiform {\it LCA}, each with underlying
	$\C[\partial]$-module $\bigoplus\limits_{i=1}^n\C[\partial]e_i$.
	
	\begin{itemize}
		\item[(i)]
		The algebra $\mathcal A_n$ is defined by
		$$[e_1{}_\lambda e_1]=(\partial+2\lambda)e_2,\quad [e_1{}_\lambda e_i]=e_{i+1}, \quad 2\leq i<n.$$
		
		\item[(ii)]
		For $n\geq3$ and $\beta\in\C$, the algebra $\mathcal C_n(\beta)$
		is defined by
		\begin{equation}\label{to:regular-family}
			[e_i{}_\lambda e_j]=\bigl(i\partial+(i+j)\lambda+(j-i)\beta\bigr)e_{i+j},\quad 2\leq i+j\leq n.
		\end{equation}
		
		\item[(iii)]
		For odd $n\geq5$ and $\gamma\in\C$, the algebra $\mathcal D_n(\gamma)$
		is defined by
		\begin{equation}\label{to:exceptional-family}
			\begin{array}{lllllllllllll}
				[e_1{}_\lambda e_1] &=&(\partial+2\lambda)e_2, &[e_1{}_\lambda e_i]&=&e_{i+1}, &2\leq i\leq n-2,\\[1mm]
				[e_1{}_\lambda e_{n-1}]&=&(\partial+n\lambda+\gamma)e_n,&[e_i{}_\lambda e_{n-i}]&=&(-1)^i n e_n, &2\leq i\leq n-2.
			\end{array}
		\end{equation}
	\end{itemize}
	In each family, the remaining brackets between basis elements
	are determined by conformal skew-symmetry or are zero. 
	
	We do not use the notation $\mathcal B$ for these families, since this notation is used for Block-type {\it LCA}s; see \cite{SYY,SXL,WangXiaLiu}.

	\begin{theorem}\label{to:null-classification}
		The algebra $\mathcal N$ admits a nonzero torus if and only if it is
		isomorphic to one of the algebras
		$\mathcal A_n$, $\mathcal C_n(\beta)$ or $\mathcal D_n(\gamma)$.
	\end{theorem}
	
	\begin{proof}
		Each algebra in the statement is totally saturated null-filiform,
		and its structure coefficients satisfy \eqref{eq5.71} with $w_i=i$. Therefore, $d_\lambda e_i=i e_i$ defines a conformal derivation, and $\C[\partial]d$ is a nonzero torus.
		
		Conversely, suppose that $\mathcal N$ admits a nonzero torus.
		By Proposition~\ref{to:null-torus-existence}, there is an adapted basis
		$\{e_1,\ldots,e_n\}$ in which
		$$[e_i{}_\lambda e_j]=\bigl(c_{ij}(i\partial+(i+j)\lambda)+b_{ij}\bigr)e_{i+j},\qquad i+j\leq n.$$
		Skew-symmetry gives $c_{ij}=c_{ji}$ and $b_{ij}=-b_{ji};$
		in particular, $b_{ii}=0$. Moreover, the image of
		$[e_1{}_\lambda e_i]$ in $\mathcal N^{i+1}/\mathcal N^{i+2}$ is nonzero, so $(c_{1i},b_{1i})\neq(0,0)$ for all $1\leq i<n.$
		Since $b_{11}=0$, we have $c_{11}\neq0$ and may normalize
		$c_{11}=1$.
		
		We first observe that the brackets $[e_1{}_\lambda e_j]$
		determine the entire multiplication table. Indeed, the Jacobi
		identity for $(e_1,e_i,e_j)$ gives
		$$\bigl(c_{1i}(i\lambda-\mu)+b_{1i}\bigr)
		[e_{i+1}{}_{\lambda+\mu}e_j]=[e_1{}_\lambda[e_i{}_\mu e_j]]
		-[e_i{}_\mu[e_1{}_\lambda e_j]].
		$$
		The polynomial multiplying the bracket on the left is nonzero.
		Since $\mathcal N$ is free over $\C[\partial]$, multiplication
		by this polynomial is injective. Thus, once the first row and
		the brackets $[e_i{}_\mu e_j]$ are known, the identity uniquely
		determines the brackets $[e_{i+1}{}_\nu e_j]$.
		Induction on $i$ proves the assertion.
		
		For $i+j+k\leq n$, comparison of the coefficients of
		$\partial^2$ and $\lambda^2$ in the relation obtained from the Jacobi identity for
		$(e_i,e_j,e_k)$ yields
		\begin{equation}\label{to:null-associativity}
			c_{ij}c_{i+j,k}
			=
			c_{jk}c_{i,j+k}
			=
			c_{ik}c_{j,i+k}.
		\end{equation}
		Taking $i=j=1$ and using $c_{11}=1$, we obtain
		\begin{equation}\label{eq6.4}
			c_{2k}=c_{1k}c_{1,k+1}.
		\end{equation}
		Similarly, for $p+q+1\leq n$, comparison of the coefficients
		of $\partial$, $\lambda$, and the constant term in the Jacobi
		identity for $(e_1,e_p,e_q)$ gives
		\begin{equation}\label{to:null-first-row-identities}
			\begin{aligned}
				b_{pq}c_{1,p+q}+p c_{pq}b_{1,p+q}
				&=
				p b_{1q}c_{p,q+1}+c_{1q}b_{p,q+1}\\
				&\quad +(p+1)b_{1p}c_{p+1,q},\\[2pt]
				(p+q+1)b_{pq}c_{1,p+q}+p c_{pq}b_{1,p+q}
				&=
				(q+1)c_{1q}b_{p,q+1}\\
				&\quad +(p+q+1)b_{1p}c_{p+1,q}
				+p c_{1p}b_{p+1,q},\\[2pt]
				b_{pq}b_{1,p+q}
				&=
				b_{1q}b_{p,q+1}+b_{1p}b_{p+1,q}.
			\end{aligned}
		\end{equation}
		In particular, setting $p=1$ in the second identity gives
		\begin{equation}\label{to:null-second-row}
			b_{2q}
			=
			(q+2)b_{1q}c_{1,q+1}
			-
			q c_{1q}b_{1,q+1}.
		\end{equation}
		
		If $n=2$, the only nonzero bracket is $[e_1{}_\lambda e_1]=(\partial+2\lambda)e_2,$ and hence $\mathcal N\simeq\mathcal A_2$. We therefore assume that $n\geq3$ and distinguish two cases.
		
		\medskip
		\noindent
		\emph{Case 1: $c_{12}\neq0$.}
		We claim that $c_{1i}\neq0$ for every $1\leq i<n$.
		Otherwise, let $r\geq3$ be the least index for which $c_{1r}=0$.
		Successively rescaling $e_3,\ldots,e_r$, we may arrange that
		$c_{1i}=1$ for all $1\leq i<r.$ Applying \eqref{to:null-associativity} to $(1,i-1,j)$ gives $c_{1,i-1}c_{ij}=c_{i-1,j}c_{1,i+j-1}.$
		For $i+j\leq r$, induction on $i$ therefore yields $c_{ij}=1$.
		For $i+j=r+1$, the same identity gives $c_{ij}=0$.
		Thus
		\[
		c_{ij}=
		\begin{cases}
			1,& i+j\leq r,\\
			0,& i+j=r+1.
		\end{cases}
		\]
		Equation~\eqref{to:null-second-row}, with $q=r-1$, now gives
		$b_{2,r-1}=-(r-1)b_{1r},$ 
		whereas the first identity in
		\eqref{to:null-first-row-identities}, with $(p,q)=(2,r-2)$,
		gives $b_{2,r-1}=2b_{1r}.$ Hence $(r+1)b_{1r}=0$, contrary to
		$(c_{1r},b_{1r})\neq(0,0)$. This proves the claim. Therefore, we may consequently rescale the basis successively so that
		$c_{1i}=1$ for all $1\leq i<n.$ Equation~\eqref{to:null-associativity} then implies $c_{ij}=1,$ $i+j\leq n.$
		
		Set $\beta=b_{12}$. By induction on $r$ we show that $b_{1r}=(r-1)\beta$ for $2\leq r<n.$
		The assertion is immediate for $r=2$. If $n\geq4$,
		equation~\eqref{to:null-second-row} with $q=2$ gives 
		$0=b_{22}=4b_{12}-2b_{13},$ so $b_{13}=2\beta$.
		
		Now let $4\leq r<n$ and assume the assertion holds for all
		smaller indices. Equation~\eqref{to:null-second-row} gives
		$$b_{2,r-2}=(r-4)\beta, \quad b_{2,r-1}=(r+1)(r-2)\beta-(r-1)b_{1r}.$$
		On the other hand, the first identity in
		\eqref{to:null-first-row-identities}, with $(p,q)=(2,r-2)$,
		yields
		$$\begin{aligned}
			b_{2,r-1}
			&=b_{2,r-2}+2b_{1r}-2b_{1,r-2}-3b_{12}\\
			&=2b_{1r}-(r+1)\beta.
		\end{aligned}
		$$
		Comparing these expressions gives $b_{1r}=(r-1)\beta$,
		completing the induction.
		
		It remains to determine $b_{ij}$. Since $b_{1j}=(j-1)\beta,$ suppose that $b_{ij}=(j-i)\beta$ for all $j$ with $i+j\leq n$. Substituting this into the second identity in \eqref{to:null-first-row-identities} gives
		$i b_{i+1,j}=i(j-i-1)\beta,$ for $i+j+1\leq n.$
		Induction on $i$ therefore gives $b_{ij}=(j-i)\beta$ for $i+j\leq n.$
		Thus the multiplication table is \eqref{to:regular-family}, and $\mathcal N\simeq\mathcal C_n(\beta)$.
		
		\medskip
		\noindent
		\emph{Case 2: $c_{12}=0$.}
		Suppose first that $c_{1i}=0$ for every $2\leq i<n$.
		Then $b_{1i}\neq0$ and we may assume that $b_{1i}=1$, $2\leq i<n.$
		The first row is therefore that of $\mathcal A_n$, and uniqueness implies $\mathcal N\simeq\mathcal A_n$.
		
		Otherwise, let $r\geq3$ be the least index such that
		$c_{1r}\neq0$. Then $c_{1i}=0$ and $b_{1i}\neq0$ for $2\leq i<r.$ 
		After successive rescaling, we may assume that $b_{1,i}=1.$ 
		The brackets of total degree at most $r$ agree with those of
		$\mathcal A_r$. In particular, $c_{ij}=b_{ij}=0$ for $2\le i,j$ and $i+j\le r$

		Write $c:=c_{1r}\neq0$. At total degree $r+1$,
		equations~\eqref{eq6.4} and \eqref{to:null-second-row} give
		$c_{2,r-1}=0$ and $b_{2,r-1}=(r+1)c.$ 
		For $2\leq p\leq r-2$, put $q=r-p$.
		The second and third identities in
		\eqref{to:null-first-row-identities} reduce to $c_{p+1,q}=0$ and $b_{p+1,q}=-b_{p,q+1},$ respectively. Consequently,
		\[
		c_{i,r+1-i}=0,
		\qquad
		b_{i,r+1-i}=(-1)^{i-2}(r+1)c,
		\qquad 2\leq i\leq r-1.
		\]
		If $r$ were odd, taking $i=(r+1)/2$ would contradict
		$b_{ii}=0$. Hence $r$ is even, and in particular $r\geq4$.
		
		We claim that $r=n-1$. Suppose instead that $r+2\leq n$. Equation~\eqref{eq6.4} gives $c_{2r}=c c_{1,r+1}.$ The first two identities in \eqref{to:null-first-row-identities}, with $(p,q)=(2,r-1)$, then yield
		\[
		3c_{3,r-1}=(r-1)c c_{1,r+1},
		\qquad
		c_{3,r-1}=(r+1)c c_{1,r+1}.
		\]
		Since $c\neq0$, these equalities force $c_{1,r+1}=0$.
		Writing $\theta=b_{1,r+1}$, equation~\eqref{to:null-second-row}
		therefore becomes $b_{2r}=-r c\theta.$
		
		For $3\leq i\leq r$, apply the third identity in
		\eqref{to:null-first-row-identities} with $(p,q)=(i-1,r+2-i)$. Since both $p$ and $q$ lie between $2$ and $r-1$, so
		$b_{1p}=b_{1q}=1$. Using the coefficients already determined
		at total degree $r+1$, we obtain 
		$b_{i,r+2-i}=(-1)^{i-3}(r+1)c\theta - b_{i-1,r+3-i}.$
		Starting from $b_{2r}=-r c\theta$, induction gives
		\[
		b_{i,r+2-i}
		=
		(-1)^{i-1}\bigl((i-1)r+i-2\bigr)c\theta,
		\qquad 2\leq i\leq r.
		\]
		Since $r$ is even, we may take $i=(r+2)/2$. The left-hand
		side is then $b_{ii}=0$, which implies $\theta=0$. Together with $c_{1,r+1}=0$, this contradicts
		$(c_{1,r+1},b_{1,r+1})\neq(0,0)$. The claim follows.
		
		We have therefore shown that $r=n-1$ and that $n$ is odd.
		Finally, replacing $e_n$ by $c e_n$ puts the multiplication
		table in the form \eqref{to:exceptional-family}, with
		$\gamma=\frac{b_{1r}}{c}.$ Hence $\mathcal N\simeq\mathcal D_n(\gamma)$, completing
		the proof.
	\end{proof}

	\begin{lemma}\label{to:null-maximal}
		Every nonzero torus of $\mathcal N$ is contained in a unique maximal torus of rank one.
	\end{lemma}
	\begin{proof}
		Let $\mathcal T$ be a nonzero torus and fix $0\ne t\in\mathcal T$. By Proposition~\ref{nf:torus-saturation}, there is an adapted basis in which $t_\lambda v_i=i\,a(\lambda)v_i,$ where $a\ne0.$
		As in the proof of Proposition~\ref{to:null-torus-existence},
		$d_\lambda v_i=i v_i$ defines a conformal derivation.
		Then $t=a(-\partial)d$.
		
		Let $s\in\operatorname{CDer}(\mathcal N)$ commute with $t$,
		and write $s_\mu v_j=\sum_i s_{ij}(\partial,\mu)v_i.$ 
		For $i\ne j$, comparison of the leading coefficients in $\partial$ in \eqref{to:commuting} gives $s_{ij}=0$.
		For $i=j$, cancellation yields $s_{ii}(\partial+\lambda,\mu)=s_{ii}(\partial,\mu),$ so $s_{ii}$ is independent of $\partial$. Consequently, $s_\mu v_i=b_i(\mu)v_i$.
		The same argument as in the proof of Proposition~\ref{nf:torus-saturation} then gives $b_i=i b_1$, whence $s=b_1(-\partial)d$.
		Conversely, every element of $\C[\partial]d$ commutes with $t$.
		
		Since $\mathcal T$ is abelian, it is contained in $\C[\partial]d$, which is itself a torus of rank one.
		Moreover, every torus containing $\mathcal T$ is contained in $\C[\partial]d$. Thus, $\C[\partial]d$ is the unique maximal torus containing $\mathcal T$.
	\end{proof}
	
	Each maximal torus determines an adapted grading of $\mathcal N$ and hence an isomorphism $\operatorname{gr}(\mathcal N)\to\mathcal N$. Comparing the isomorphisms associated with two maximal tori yields the following conjugacy result.
	\begin{proposition}\label{to:null-conjugacy}
		Suppose that $\mathcal N$ admits a nonzero torus. Then any two maximal tori of $\mathcal N$ are conjugate by an automorphism of the form $\exp(D)$, where $D\in\operatorname{Der}(\mathcal N)$ is nilpotent.
	\end{proposition}
	\begin{proof}
		Let $\mathcal T_1,\mathcal T_2$ be maximal tori. By the proof of Lemma~\ref{to:null-maximal}, for $r=1,2$ there is an adapted basis $v_{r,1},\ldots,v_{r,n}$ such that $\mathcal T_r=\C[\partial]d_r,\ (d_r)_\lambda v_{r,i}=i v_{r,i}.$
		Set $\mathcal N_{r,i}=\C[\partial]v_{r,i}$. By Corollary~\ref{to:weight-decomposition}, $[\mathcal N_{r,i}\cdot\mathcal N_{r,j}]\subseteq\mathcal N_{r,i+j},$ where $\mathcal N_{r,k}=0$ for $k>n$.
		
		Since the basis is adapted, the class $\bar v_{r,i}=v_{r,i}+\mathcal N^{i+1}$ is a module basis of $\mathcal N^i/\mathcal N^{i+1}$. Thus $\bar v_{r,1},\ldots,\bar v_{r,n}$ form a module basis of $\operatorname{gr}(\mathcal N)$, and the map $s_r(\bar v_{r,i})=v_{r,i}$ extend to a module isomorphism $s_r:\operatorname{gr}(\mathcal N)\to\mathcal N$.
		
		For $i+j\leq n$, the weight rule gives $[v_{r,i}{}_\lambda v_{r,j}]\in\C[\partial,\lambda]v_{r,i+j}.$ By the definition of the bracket on $\operatorname{gr}(\mathcal N)$, $[\bar v_{r,i}{}_\lambda\bar v_{r,j}]$ is obtained by reducing this bracket modulo $\mathcal N^{i+j+1}[\lambda]$. The map $s_r$ sends $\bar v_{r,i+j}$ to $v_{r,i+j}$ and leaves its polynomial coefficient unchanged. Hence $s_r([\bar v_{r,i}{}_\lambda\bar v_{r,j}])=[v_{r,i}{}_\lambda v_{r,j}].$ Therefore, $s_r$ is an {\it LCA} isomorphism.
		
		Set $\varphi=s_1s_2^{-1}$. Since each $s_r$ is the inverse of the natural projection on every homogeneous component, $\varphi$ preserves the lower central filtration and induces the identity on each quotient $\mathcal N^i/\mathcal N^{i+1}$. Moreover, $\varphi(\mathcal N_{2,i})=\mathcal N_{1,i}.$ As both $\mathcal N_{2,i}$ and $\mathcal N_{1,i}$ are free of rank one and their images in $\mathcal N^i/\mathcal N^{i+1}$ are module bases, there exists $c_i\in\C^\times$ such that $\varphi(v_{2,i})=c_i v_{1,i}.$ Since $(d_r)_\lambda v_{r,i}=i v_{r,i}$ for $r=1,2$, we have
		$$\varphi\bigl((d_2)_\lambda v_{2,i}\bigr)=i c_i v_{1,i}=(d_1)_\lambda\varphi(v_{2,i})$$
		for every $i$. Hence $\varphi d_2=d_1\varphi,$ and $\varphi d_2\varphi^{-1}=d_1,$ and therefore $\varphi\mathcal T_2\varphi^{-1}=\mathcal T_1.$
		
		Set $Q=\varphi-\operatorname{id}.$ Then we have $Q(\mathcal N^i)\subseteq\mathcal N^{i+1},$ and hence $Q^n=0$. Define
		$$D=\log\varphi=\sum\limits_{k=1}^{n-1}\frac{(-1)^{k-1}}{k}Q^k.$$
		Then $D$ is nilpotent.
		
		For $z\in \mathbb Z_{\geq0}$, $\varphi^z=(\operatorname{id}+Q)^z=\sum\limits_{k=0}^{n-1}\binom{z}{k}Q^k.$ For fixed $u,v\in\mathcal N$, both sides of
		\begin{equation}\label{eq6.9.1}
			\varphi^z([u_\lambda v])=[\varphi^z(u)_\lambda\varphi^z(v)]
		\end{equation}
		are polynomial in $z$. Since the identity holds for every $z\in\mathbb Z_{\geq0}$, it holds for all $z\in\C$.
		
		Since for $k\geq1$ we have 
		$$\left.\frac{d}{dz}\binom{z}{k}\right|_{z=0}=\frac{(-1)^{k-1}}{k}.$$
		Then we obtain 
		$$\left.\frac{d}{dz}\varphi^z\right|_{z=0}=\sum\limits_{k=1}^{n-1}\frac{(-1)^{k-1}}{k}Q^k=D.$$
		Differentiating the identity \eqref{eq6.9.1} at $z=0$ and using $\varphi^0=\operatorname{id}$ gives $D([u_\lambda v])=[D(u)_\lambda v]+[u_\lambda D(v)].$ Thus, $D\in\operatorname{Der}(\mathcal N)$ and $\exp(D)=\varphi.$ 
	\end{proof}
	
	The conjugacy of maximal tori reduces the isomorphism problem for the three families to rescalings of their standard basis vectors. A comparison of the defining brackets under these rescalings completes the classification in Theorem~\ref{to:null-classification}.
	
	\begin{theorem}\label{to:null-isomorphism}
		The algebras $\mathcal A_n$, $\mathcal C_n(\beta)$, and
		$\mathcal D_n(\gamma)$ are pairwise nonisomorphic.
		Moreover, within each of the families $\mathcal C_n(\beta)$ and $\mathcal D_n(\gamma)$, the parameter uniquely determines the isomorphism class.
	\end{theorem}
	
	\begin{proof}
		Let $\varphi:\mathcal N\to\mathcal N'$ be an isomorphism between two of the listed algebras, and let $d_\lambda e_i=i e_i$ and $ d'_\lambda e_i'=i e_i'$ be their standard conformal derivations. By Lemma~\ref{to:null-maximal} and Prposition~\ref{to:null-conjugacy}, there exists $\psi\in\operatorname{Aut}(\mathcal N')$ such that 
		$$\psi\varphi(\C[\partial]d)\varphi^{-1}\psi^{-1}=\C[\partial]d'.$$
		Replacing $\phi$ by $\psi\varphi$, then we obtain
		$\phi(\C[\partial]d)\phi^{-1}=\C[\partial]d'.$ Hence $\varphi d\varphi^{-1}=\gamma d'$ for some $\gamma\in\C^\times$. Applying this equality to $\mathcal N'_1$ and evaluating at $\lambda=0$ gives $\gamma=1$, since both $d$ and $d'$ act as the identity on $\mathcal N_1$ and $\mathcal N'_1$, respectively. Thus $\varphi d=d'\varphi.$
		
		Write $\varphi(e_i)=\sum\limits_jp_{ji}(\partial)e_j'.$ Applying $\varphi d=d'\varphi$ to $e_i$ gives $i\,p_{ji}(\partial)=j\,p_{ji}(\partial+\lambda).$ It follows that $p_{ji}=0$ for $j\ne i$, while $p_{ii}\in\C^\times$. Therefore $\varphi(e_i)=h_i e_i'$ with $h_i\in\C^\times.$
		
		We can now compare the brackets $[e_1{}_\lambda e_i]$ directly. The coefficient of $\partial$ is nonzero only for $i=1$ in $\mathcal A_n$, for every $1\leq i<n$ in $\mathcal C_n(\beta)$, and precisely for $i=1,n-1$ in $\mathcal D_n(\gamma)$. Since $\varphi(e_i)=h_i e_i'$, this vanishing pattern is preserved by isomorphisms. Hence the three families are pairwise nonisomorphic whenever they are defined.
		
		Suppose next that $\mathcal C_n(\beta)\simeq\mathcal C_n(\gamma)$. Comparing $[e_1{}_\lambda e_2]$ gives $(\partial+3\lambda+\beta)h_3=(\partial+3\lambda+\gamma)h_1h_2.$ Comparison of the coefficients of $\partial$ gives $h_3=h_1h_2$, and then comparison of the constant terms yields $\beta=\gamma$.
		
		Similarly, if $\mathcal D_n(\beta)\simeq\mathcal D_n(\gamma)$, comparison of $[e_1{}_\lambda e_{n-1}]$ gives $(\partial+n\lambda+\beta)h_n=(\partial+n\lambda+\gamma)h_1h_{n-1},$ and hence $\beta=\gamma$. The converse in both cases follows immediately from the defining brackets.
	\end{proof}
	
	\section{Solvable extensions with prescribed nilradicals}\label{ex:section}
	
	In this section we extend the classical method for constructing solvable Lie algebras with prescribed nilradicals to the conformal setting. We then determine, up to isomorphism, the unique maximal solvable extension of each algebra classified in Section~\ref{nf:section}.
	
	\begin{proposition}\label{pr:nilradical-saturated}
		The nilradical and the solvable radical of $\mathcal L$
		are saturated in $\mathcal L$.
	\end{proposition}
	\begin{proof} Let $\mathcal N$ be the nilradical of $\mathcal L$ and set $\mathcal S=\operatorname{Sat}_{\mathcal L}(\mathcal N)$.
		Taking $\mathcal U=\mathcal L$ in Lemma~\ref{nf:saturation-brackets}, we obtain
		$[\mathcal L\cdot\mathcal S]=[\mathcal L\cdot\mathcal N]\subseteq\mathcal N \subseteq\mathcal S.$
		Hence $\mathcal S$ is an ideal of $\mathcal L$.
		Two further applications of Lemma~\ref{nf:saturation-brackets} give
		\[\mathcal S^2 = [\mathcal S\cdot\mathcal S] = [\mathcal N\cdot\mathcal S] = [\mathcal N\cdot\mathcal N] = \mathcal N^2.
		\]
		Applying this as the base of induction one can prove $\mathcal S^j=\mathcal N^j$ for every $j\geq2$. 
		Since $\mathcal N$ is nilpotent, so is $\mathcal S$.
		Consequently, $\mathcal S\subseteq\mathcal N$ by the maximality of the nilradical. The opposite inclusion is immediate,
		and hence $\mathcal S=\mathcal N$.
		
		An analogous argument shows that the solvable radical $\mathcal R$ of $\mathcal L$ is saturated in $\mathcal L.$
	\end{proof}
	
	Let $\mathcal N$ be the nilradical of $\mathcal R$. By Proposition~\ref{pr:nilradical-saturated}, $\mathcal R/\mathcal N$ is torsion-free and hence free. Moreover, \cite[Corollary~8.3]{DK} states that $\mathcal R^2$ is nilpotent, so $\mathcal R^2\subseteq\mathcal N$. Therefore, we may write 
	$\mathcal R=\mathcal N\oplus_{\C[\partial]}\mathcal Q$ for some submodule $\mathcal Q\subseteq\mathcal R$. 
	
	\begin{theorem}\label{ex:complement-nonnilpotent}
		For every nonzero $x\in\mathcal Q$, the restriction
		${\operatorname{ad}_x}_{| \mathcal N}$ is nonnilpotent.
	\end{theorem}
	\begin{proof} Suppose that $(\operatorname{ad}_x)_{|\mathcal N}$ is nilpotent for some $0\ne x\in\mathcal Q$, and set $\mathcal V=\mathcal N\oplus_{\C[\partial]}\C[\partial]x.$ Since $\mathcal R^2\subseteq\mathcal N$, the submodule $\mathcal V$ is a solvable ideal of $\mathcal R$. By the nilpotence of $\mathcal N$ and the assumption on $x$, every $b\in\mathcal N\cup\{x\}$ acts nilpotently on $\mathcal N$. Moreover, 
		$[b_\lambda\mathcal V]\subseteq\mathcal N[\lambda]$,
		so its adjoint action on $\mathcal V$ is also nilpotent.
		
		Set $\overline{\mathcal V}=\mathcal V/\operatorname{Tor}(\mathcal V).$ Then the adjoint action makes $\overline{\mathcal V}$ a free conformal $\mathcal V$-module. Applying \eqref{ca:triangular-character}, we obtain a flag
		$$0=\overline{\mathcal V}_0\subset\overline{\mathcal V}_1\subset\cdots\subset\overline{\mathcal V}_m=\overline{\mathcal V}$$
		whose successive quotients are free of rank one. Choose a generator $\bar v_i$ of
		$\overline{\mathcal V}_i/\overline{\mathcal V}_{i-1}$, and write $b_\lambda\bar v_i=\phi_i(b)(\lambda)\bar v_i,$ $b\in\mathcal V.$
		For $b\in\mathcal N\cup\{x\}$ we have 
		$$0=b_{\lambda_q}\cdots b_{\lambda_1}\bar v_i=\left(\prod_{j=1}^q\phi_i(b)(\lambda_j)\right)\bar v_i$$
		for some $q\geq1$. Since this quotient is free of rank one, $\phi_i(b)=0$, which yields $b_\lambda\overline{\mathcal V}_i\subseteq\overline{\mathcal V}_{i-1}[\lambda]$ for all $b\in\mathcal V.$ Hence, $\mathcal V^{m+1}\subseteq\operatorname{Tor}(\mathcal V).$ The centrality of $\operatorname{Tor}(\mathcal V)$ now gives $\mathcal V^{m+2}=0$. Therefore $\mathcal V$ is a nilpotent ideal of $\mathcal R$ properly containing its nilradical $\mathcal N$, a contradiction.
	\end{proof}

	Theorem~\ref{ex:complement-nonnilpotent} motivate to introduce the the following definition.
	\begin{definition}
		Conformal derivations $d_1,\ldots,d_s$ of an {\it LCA}
		$\mathcal N$ are called \emph{nil-independent} if no nontrivial
		$\C[\partial]$-linear combination of them is nilpotent.
	\end{definition}
	
	Note that the inclusion $\mathcal R^2\subseteq\mathcal N$, together with the nilpotence of $\mathcal N$ and the identity
	$$[\operatorname{ad}(x)_\lambda\operatorname{ad}(y)]_{\lambda+\mu}=\operatorname{ad}([x_\lambda y])_{\lambda+\mu},$$
	implies that every element of $\operatorname{ad}([\mathcal Q\cdot\mathcal Q])$ is nilpotent.
	
	From Theorem~\ref{ex:complement-nonnilpotent}, it follows that the rank of $\mathcal Q$ does not exceed the maximal number of nil-independent conformal derivations of $\mathcal N$ whose pairwise commutators are nilpotent. This observation shows that the method for constructing finite-dimensional solvable Lie algebras with a prescribed nilpotent algebra, extends naturally to the setting of solvable extensions of a given nilpotent {\it LCA} of finite rank. More precisely, the construction method consists of the following steps:
	\begin{itemize}
		\item[-]describe the space $\operatorname{CDer}(\mathcal N)$ for a given nilpotent {\it LCA} $\mathcal N$;
		
		\item[-]determine all nil-independent conformal derivations whose pairwise commutators are nilpotent, say $d_1,\ldots,d_k$, such that $[d_i{}_\lambda d_j]$ is nilpotent for all $1\leq i,j\leq k$;
		
		\item[-]define $\mathcal Q=\operatorname{Span}_{\C[\partial]}\{x_1,\ldots,x_s\}$ by the actions $(\operatorname{ad}_{x_i})_{|\mathcal N}=d_{j_i}$, with $1\leq j_1<\cdots<j_s\leq k$, giving the product $[\mathcal N\cdot\mathcal Q]$;
		
		\item[-]apply conformal skew-symmetry and the Jacobi identity, together with suitable changes of basis in $\mathcal N\bigoplus\mathcal Q$, to determine the product $[\mathcal Q\cdot\mathcal Q]$.
	\end{itemize}
	
	For the construction and classification of solvable Lie algebras
	with a prescribed nilradical, see, for example,
	\cite{Ancochea2006, Campoamor2010, OS, Snobl2010}
	and the references therein.
	
	\begin{definition}
		Let $\mathcal N$ be a finite nilpotent {\it LCA}. A finite solvable
		{\it LCA} $\mathcal R$ is called a \emph{solvable extension} of
		$\mathcal N$ if $\operatorname{Nil}(\mathcal R)=\mathcal N$.
		It is called \emph{maximal} if there is no proper embedding
		$\mathcal R\hookrightarrow\mathcal S$ into another solvable extension
		$\mathcal S$ of $\mathcal N$ whose restriction maps
		$\mathcal N$ onto itself.
	\end{definition}

	The following results establish the uniqueness of the maximal solvable extension of a totally saturated null-filiform {\it LCA} admitting a nonzero torus, in analogy with the Lie algebra case proved in \cite{OS}.
	
	\begin{theorem}\label{ex:maximal-extension}
		Let $\mathcal N$ be a totally saturated null-filiform {\it LCA} that admits a nonzero torus. Then, up to isomorphism, it has a unique maximal solvable extension $\mathcal R_{\mathcal N}=\mathcal T\ltimes\mathcal N,$ where $\mathcal T=\C[\partial]d$ is a maximal torus of $\mathcal N$, acting in a suitable adapted basis $\{e_1,\ldots,e_n\}$ by $d_\lambda e_i=i e_i,$ $1\leq i\leq n.$
	\end{theorem}
	
	\begin{proof}
		Let $\mathcal S$ be a finite solvable Lie conformal algebra with nilradical $\mathcal N$. Since the torsion submodule of a Lie conformal algebra is central, we have
		$$\operatorname{Tor}_{\C[\partial]}(\mathcal S)\subseteq Z(\mathcal S)\subseteq \operatorname{Nil}(\mathcal S)=\mathcal N.$$
		As $\mathcal N$ is free, it follows that $\mathcal S$ is torsion-free and hence free over $\C[\partial]$.
		
		Let $\mathcal S$ be a solvable extension of $\mathcal N$. Choose a generator $\bar e$ of $\mathcal N/\mathcal N^2$. For any $x\in\mathcal{S}$, the induced action has the form $x_\lambda\bar e=a_x(\lambda)\bar e.$
		Since $\mathcal N$ acts trivially on this quotient, we obtain
		a $\C[\partial]$-linear map
		$$\rho_{\mathcal S}\colon\mathcal S/\mathcal N\longrightarrow\C[\lambda],\qquad x+\mathcal N\longmapsto a_x(\lambda).$$

		Write $\mathcal S=\mathcal N\oplus_{\C[\partial]}\mathcal Q$.
		If $x\in\mathcal Q$ satisfies $a_x=0$, then
		$(\operatorname{ad}_x)_{|\mathcal N}$ is nilpotent by
		Proposition~\ref{pr:nilpotence-quotient}, so
		Theorem~\ref{ex:complement-nonnilpotent} gives $x=0$.
		Thus $\rho_{\mathcal S}$ is injective, and consequently $\operatorname{rank}(\mathcal S/\mathcal N)\leq1.$
		
		By Proposition~\ref{to:null-torus-existence}, choose an adapted
		basis $\{e_1, \ldots, e_n\}$ such that $d_\lambda e_i=i e_i$ defines a conformal derivation and by Lemma~\ref{to:null-maximal} we get that $\mathcal T=\C[\partial]d$ is a maximal torus of $\mathcal N$.
		
		The semidirect product $\mathcal R_{\mathcal N}=\mathcal T\ltimes\mathcal N$ is solvable. It satisfies the conditions
		$\mathcal R_{\mathcal N}^2=\mathcal N$ and $[\mathcal R_{\mathcal N}\cdot\mathcal N]=\mathcal N.$ Since $\operatorname{Nil}(\mathcal R_{\mathcal N})$ is saturated by Proposition~\ref{pr:nilradical-saturated}, and
		$\mathcal R_{\mathcal N}/\mathcal N$ has rank one, we obtain 
		$\mathcal N=\operatorname{Nil}(\mathcal R_{\mathcal N})$.
		
		Suppose that $\mathcal R_{\mathcal N}\subseteq\mathcal S$,
		where $\mathcal S$ is a free solvable extension with nilradical
		$\mathcal N$. Since $d_\lambda\bar e_1=\bar e_1$, the restriction
		of $\rho_{\mathcal S}$ to $\mathcal R_{\mathcal N}/\mathcal N$ has image $\C[\lambda]$. The injectivity of $\rho_{\mathcal S}$ therefore gives $\mathcal S=\mathcal R_{\mathcal N}$. Thus, $\mathcal R_{\mathcal N}$ is maximal.

		Now let $\mathcal S$ be any maximal solvable extension of $\mathcal N$. Since $\mathcal N$ admits the proper extension $\mathcal R_{\mathcal N}$, we have $\operatorname{rank}(\mathcal S/\mathcal N)=1$.
		Choose $x\in\mathcal S$ whose image generates $\mathcal S/\mathcal N$. In an adapted basis $\{e_1, \ldots, e_n\}$, write
		$[x_\lambda e_i]\equiv a_i(\lambda)e_i\pmod{\mathcal N^{i+1}[\lambda]}.$
		The injectivity of $\rho_{\mathcal S}$ gives $a_1\ne0$. As in the proof of Proposition~\ref{nf:torus-saturation} we obtain $a_i=i\,a,$ where $a:=a_1\ne0.$
		
		Let $\mathcal H$ be a Cartan subalgebra of $\mathcal S$. In the adjoint ideal flag
		$$0=\mathcal N^{n+1}\subset\mathcal N^n\subset\cdots
		\subset\mathcal N\subset\mathcal S,$$
		the characters on $\mathcal N^i/\mathcal N^{i+1}$ take the
		nonzero values $ia$ at $x$, whereas the character on
		$\mathcal S/\mathcal N$ is zero.
		Proposition~\ref{ca:cartan-rank} gives
		$\operatorname{rank}(\mathcal H)=1$.
		Since $\mathcal S^2\subseteq\mathcal N$,
		Corollary~\ref{ca:cartan-derived} yields
		$\mathcal S=\mathcal H\oplus_{\C[\partial]}\mathcal N.$
		Moreover,
		$[\mathcal H\cdot\mathcal H]\subseteq\mathcal H\cap\mathcal N=0$.
		
		Choose a generator $y$ of $\mathcal H$ such that
		$y+\mathcal N=x+\mathcal N$.
		The induced action of $y$ on
		$\mathcal N^i/\mathcal N^{i+1}$ has weight $ia$.
		These weights are pairwise distinct, so the generalized weight
		decomposition splits the lower central filtration into free
		rank-one $\mathcal H$-modules. Thus there is an adapted basis
		$\{v_1, \ldots, v_n\}$ such that $[y_\lambda v_i]=i\,a(\lambda)v_i.$ 
		Cancelling $a(\lambda)$ in the derivation identity shows that
		$d'_\lambda v_i=i v_i$ defines a conformal derivation.
		By Lemma~\ref{to:null-maximal} and the proof of
		Proposition~\ref{to:null-conjugacy}, there exists
		$\psi\in\operatorname{Aut}(\mathcal N)$ such that
		$\psi d'\psi^{-1}=d.$ 
		Consequently, the map defined as 
		$y\mapsto a(-\partial)d $ and $z\mapsto\psi(z)$ for $z\in\mathcal N$ embeds $\mathcal S$ into $\mathcal R_{\mathcal N}$ and maps
		$\mathcal N$ onto itself. By maximality, this embedding is
		surjective. Hence $a\in\C^\times$ and
		$\mathcal S\simeq\mathcal R_{\mathcal N}$.
	\end{proof}
	
	\
	
	\bigskip
	\noindent\textbf{Acknowledgements.}
	The authors thank Professor Qing Wang for helpful discussions
	and suggestions on this work.
	\\[3mm]
	{\bf Funding} {{\small{\ The first author was supported by Basic Research Program of Jiangsu (BK20251784).}}}
	\\[3mm]
	{\bf Data availability} {{\small{\ Data sharing not applicable to this article as no datasets were generated or analysed during
				the current study.}}
		\\[3mm]
		{\bf Conflict of interest} {{\small{\ The authors have no competing interests to declare that are relevant to the content of this
					article.}}
			\\[3mm]
			{\bf AI Declaration} {{\small{\ 
						During the preparation of this manuscript, the authors used AI solely to improve the language, grammar, and readability of the text and to assist with LaTeX typesetting. The tool was not used to generate mathematical results or construct their proofs. The authors reviewed and edited all AI-assisted output and take full responsibility for the final manuscript.}}

			\end{document}